\documentclass{article}

\usepackage[T1]{fontenc}
\usepackage[utf8]{inputenc}

\usepackage{lmodern}

\usepackage{bm}

\usepackage[dvipsnames]{xcolor}

\usepackage{microtype}

\usepackage{amsmath}

\usepackage{amsthm}

\usepackage{amssymb}

\usepackage[showonlyrefs]{mathtools}

\usepackage{array}

\usepackage{authblk}

\usepackage[letterpaper,top=2in, bottom=1.5in, left=1in, right=1in]{geometry}

\usepackage{fancyhdr, blindtext}
\newcommand{\changefont}{\fontsize{7}{9}\selectfont}
\newcolumntype{L}[1]{>{\raggedright\let\newline\\\arraybackslash\hspace{0pt}}m{#1}}
\newcolumntype{C}[1]{>{\centering\let\newline\\\arraybackslash\hspace{0pt}}m{#1}}
\newcolumntype{R}[1]{>{\raggedleft\let\newline\\\arraybackslash\hspace{0pt}}m{#1}}

\usepackage[colorinlistoftodos]{todonotes}

\usepackage[backend=biber,style=ieee,doi=false,isbn=false,url=false]{biblatex}

\usepackage[colorlinks=true, linkcolor=blue, citecolor=blue, urlcolor=blue, pdfborder={0 0 0},bookmarksnumbered,bookmarksdepth=3]{hyperref}

\usepackage{mathtools}
\usepackage[normalem]{ulem} 

\usepackage{amsmath}
\usepackage{amssymb}
\usepackage{mathtools}
\usepackage{mathrsfs}

\newcommand{\cut}[1]{{}}

\newcommand{\EE}{{\mathbb{E}}} 				
\newcommand{\PP}{\mathbb{P}} 				
\newcommand{\RR}{\mathbb{R}}				
\newcommand{\HH}{\mathbb{H}}				
\newcommand{\NN}{\mathbb{N}}				

\newcommand{\DeclareAutoPairedDelimiter}[3]{%
	\expandafter\DeclarePairedDelimiter\csname Auto\string#1\endcsname{#2}{#3}%
	\begingroup\edef\x{\endgroup
		\noexpand\DeclareRobustCommand{\noexpand#1}{%
			\expandafter\noexpand\csname Auto\string#1\endcsname*}}%
	\x}

\DeclareAutoPairedDelimiter{\p}{(}{)} 					
\DeclareAutoPairedDelimiter{\sp}{[}{]} 					
\DeclareAutoPairedDelimiter{\abs}{|}{|} 					
\DeclareAutoPairedDelimiter{\cp}{\{}{\}} 				
\DeclareAutoPairedDelimiter{\dotp}{\langle}{\rangle} 	
\DeclareAutoPairedDelimiter{\n}{\Vert}{\Vert} 			
\DeclareAutoPairedDelimiter{\cl}{\lceil}{\rceil}

\newcommand{\cF}{{\mathcal{F}}}
\newcommand{\cG}{{\mathcal{G}}}

\newcommand{\cO}{{\mathcal{O}}}

\usepackage{amsmath}
\usepackage{amsthm}
\usepackage{amssymb}
\usepackage{mathtools}
\usepackage{mathrsfs}

\usepackage[nameinlink]{cleveref}

\newcommand{\bc}{\begin{center}}
\newcommand{\ec}{\end{center}}

\newcommand{\bdm}{\begin{displaymath}}
\newcommand{\edm}{\end{displaymath}}

\newcommand{\beq}{\begin{equation}}
\newcommand{\eeq}{\end{equation}}

\newcommand{\bfl}{\begin{flushleft}}
\newcommand{\efl}{\end{flushleft}}

\newcommand{\bt}{\begin{tabbing}}
\newcommand{\et}{\end{tabbing}}

\newcommand{\beqn}{\begin{align}}
\newcommand{\eeqn}{\end{align}}

\newcommand{\beqs}{\begin{align*}} 
\newcommand{\eeqs}{\end{align*}}  

\newtheoremstyle{Fancyplain}
{\topsep}   
{\topsep}   
{\itshape}  
{0pt}       
{\bfseries} 
{}         
{5pt plus 1pt minus 1pt} 
{\thmname{#1} \thmnumber{#2}. \thmnote{\normalfont\bfseries#3.}}

\theoremstyle{Fancyplain}
\newtheorem{thm}{Theorem}
\newtheorem{lem}{Lemma}
\newtheorem{cor}[lem]{Corollary}
\newtheorem{prop}[lem]{Proposition}

\crefname{thm}{Thm.}{Thms.}
\Crefname{thm}{Theorem}{Theorems}
\crefname{lem}{Lem.}{Lems.}
\Crefname{lem}{Lemma}{Lemmas}
\crefname{cor}{Cor.}{Cors.}
\Crefname{cor}{Corollary}{Corollaries}
\crefname{prop}{Prop.}{Props.}
\Crefname{prop}{Proposition}{Propositions}

\newtheoremstyle{Fancydefinition}
{\topsep}   
{\topsep}   
{\normalfont}  
{0pt}       
{\bfseries} 
{}         
{5pt plus 1pt minus 1pt} 
{\thmname{#1} \thmnumber{#2}. \thmnote{\normalfont\bfseries#3.}}

\theoremstyle{Fancydefinition}
\newtheorem{defn}[lem]{Definition}

\newtheorem{rem}{Remark}
\newtheorem{asmp}{Assumption}

\crefname{defn}{Defn.}{Defns.}
\Crefname{defn}{Definition}{Definitions}
\crefname{example}{Ex.}{Exs.}
\Crefname{example}{Example}{Examples}
\crefname{xca}{Ex.}{Exs.}
\Crefname{xca}{Exercise}{Exercises}
\crefname{rem}{Rem.}{Rems.}
\Crefname{rem}{Remark}{Remarks}
\crefname{asmp}{Asmp.}{Asmps.}
\Crefname{asmp}{Assumption}{Assumptions}
\crefname{section}{Sec.}{Secs.}
\Crefname{section}{Section}{Sections}

\numberwithin{equation}{section}
\numberwithin{figure}{section}

\bibliography{MyLibrary}

\newcommand{\eps}{\epsilon}
\newcommand{\del}{\delta}

\newcommand{\hx}{\hat{x}}
\newcommand{\vj}{\vec{j}}
\newcommand{\seq}[1]{{\p{#1_1,#1_2,\ldots}}}

\usepackage[labelfont=bf]{caption}

\allowdisplaybreaks

\usepackage{todonotes}
\reversemarginpar

\usepackage{algorithm,algpseudocode}
\usepackage{float}
\usepackage{subcaption}
\DeclareCaptionSubType*{algorithm}

\DeclareCaptionLabelFormat{alglabel}{Alg.~#2}

\begin{document}

\title{More Iterations per Second, Same Quality -- Why Asynchronous Algorithms may Drastically Outperform Traditional Ones}

\author{Robert Hannah\footnote{Email: \href{mailto:RobertHannah89@math.ucla.edu}{RobertHannah89@math.ucla.edu}} }
\author{Wotao Yin\footnote{Email: \href{mailto:wotaoyin@math.ucla.edu}{wotaoyin@math.ucla.edu}}}
\affil{Department of Mathematics, University of California, Los Angeles, CA 90095, USA}

\date{\today}
\maketitle

\begin{abstract}
In this paper, we consider the convergence of a very general asynchronous-parallel algorithm called ARock \cite{PengXuYanYin2016_arock}, that takes many well-known asynchronous algorithms as special cases (gradient descent, proximal gradient, Douglas Rachford, ADMM, etc.). In asynchronous-parallel algorithms, the computing nodes simply use the most recent information that they have access to, instead of waiting for a full update from all nodes in the system. This means that nodes do not have to waste time waiting for information, which can be a major bottleneck, especially in distributed systems. When the system has $p$ nodes, asynchronous algorithms may complete $\Theta(\ln\p{p})$ more iterations than synchronous algorithms in a given time period (``more iterations per second'').

Although asynchronous algorithms may compute more iterations per second, there is error associated with using outdated information. How many more iterations in total are needed to compensate for this error is still an open question. The main results of this paper aim to answer this question. We prove, loosely, that as the size of the problem becomes large, the number of additional iterations that asynchronous algorithms need becomes negligible compared to the total number (``same quality'' of the iterations). Taking these facts together, our results provide solid evidence of the potential of asynchronous algorithms to vastly speed up certain distributed computations.
\end{abstract}

\section{Introduction} \label{sec:Introduction}
Designing efficient algorithms to solve large-scale optimization problem is an increasingly important area of research. However parallel algorithms are more challenging to analyze and implement because there is a host of additional considerations and issues that only arise in parallel settings.

The vast majority of parallel algorithms are synchronous. At each iteration, all processors will compute an update, and then share this update with all others. The next iteration can proceed only when all processors are finished computing an update. This synchronization can be extremely expensive at scale or on a congested network. Network latency, packet loss, loss of a node, unexpected drains on computational resources that affect even one node will cause the entire system to slow down. Asynchronous algorithms overcome the problem of synchronization by simply computing their next update with the most recent information they have available. 
Though this will eliminate synchronization penalty, there is a slowdown or penalty associated with using outdated information. It is not immediately clear whether asynchronous versions of algorithms will be faster, or will even converge to a solution.

In this paper, we examine the theoretical performance of asynchronous algorithms compared to traditional synchronous ones. We do this by analyzing a very general asynchronous-parallel algorithm called ARock \cite{PengXuYanYin2016_arock}. ARock takes many popular algorithms as special cases, such as asynchronous block gradient descent, forward backward, proximal point, etc. Hence our results will apply to all of these algorithms.

\subsection{Main argument} \label{ssec:Main-argument}
This paper aims to provide a solid theoretical evidence that asynchronous algorithms will drastically outperform synchronous ones at scale under a wide range of scenarios that ARock encompasses. These include gradient descent, forward backward, etc. for strongly convex objectives with Lipschitz gradient, or any other algorithm that can be written in the form of a block fixed-point algorithm on a contractive operator. Our argument involves a series of steps, that together will bolster our conclusion. 

\begin{enumerate}
\item We first argue that synchronous algorithms have significant synchronization penalty with scale under a well-justified model. That is, as the number of processing nodes $p$ increases, a larger and larger portion time will be spend waiting instead of computing. In fact, it will be shown that synchronous algorithms progress at least $\Theta(\ln(p))$ fewer epochs\footnote{We measure the iteration complexity in terms of \textbf{epochs}. This is a context-dependent unit of computation which loosely correspond to one computation of $Sx$ for some operator $S$ and vector $x$, e.g., the calculation of a full gradient for gradient descent.} per second than asynchronous algorithms.

\item It has always been plausible that asynchronous algorithms may progress more epochs per second. However, it has always been an open question whether an increased number of iterations per second is worth a potential iteration complexity penalty for using outdated information. In this paper, we provide a surprising answer: the iteration complexity of ARock is asymptotically the same as the corresponding synchronous algorithm, even under certain kinds of unbounded delays.

\item Since asynchronous algorithms allow for far more iterations per second (more iterations), and these iterations make the nearly same progress as synchronous ones \emph{per iteration} (same quality), asynchronous algorithms may drastically outperform synchronous ones in large-scale applications. Since ARock is extremely general, this argument applies to wide variety of asynchronous algorithms, such as block gradient descent, proximal gradient, Douglas-Rachford, etc.

\end{enumerate}

The remainder of this section will introduce the general setting, notation, and background for our results. We will discuss related work in \Cref{ssec:Related-work}.

\subsection{Fixed-point algorithms, and their generality}
ARock is a fixed-point algorithm, which allows it to be very general. This is because most optimization problems and algorithms can be written in the fixed-point form. Many popular algorithms such as asynchronous block gradient descent, forward backward, proximal point, etc. are special cases of ARock. They differ only in their choice of fixed-point operator $T$ (see \cite{HannahYin2016_unbounded} for a list of applications and special cases). Hence our results will apply to all of these algorithms.

Take an operator $T:\HH\to\HH$ with Lipschitz constant $0<r\leq1$. Such an operator is called nonexpansive. The aim is to find a \textbf{fixed-point} of this operator: That is, a point $x^*\in\HH$ such that $Tx^*=x^*$. For example, smooth minimization of a convex function $f:\HH\to\RR$ with $L$-Lipschitz gradient $\nabla f$ is equivalent to finding a fixed point of the nonexpansive operator $T=I-\gamma\nabla f$, where $I$ is the identity, and $0<\gamma\leq\frac{2}{L}$. The set of fixed points of an operator $T$ is denoted $\text{Fix}(T)$. In this paper, we consider the case where $0<r<1$, that is, $T$ is contractive, since this leads to linear convergence\footnote{The case of $r=1$ was considered in \cite{HannahYin2016_unbounded}.}.

The most common fixed-point algorithm is the Krasnosel'ski\u{\i}-Mann (KM) algorithm. ARock is essentially an asynchronous block-coordinate version of KM iteration.

\begin{defn}[Krasnosel'ski\u{\i}-Mann algorithm]
Let $\eps>0$, and $\eta^{k}$ be a series of step lengths in $\p{\eps,1-\eps}$. Let $T$ be a nonexpansive operator with at least one fixed point, and
\begin{align}\label{eq:S_def}
S &:= I-T.
\end{align}
Staring from $x^0$, the KM Algorithm is defined by the following:
\begin{align}
\begin{aligned}
x^{k+1} &= x^{k}-\eta^{k}S(x^{k})
\end{aligned}
\end{align}
\end{defn}

\begin{rem}
\textbf{Gradient descent} is equivalent to the KM algorithm with $T=I-\gamma\nabla f$. 
If the fixed-point framework is unfamiliar, it may be helpful to mentally replace $S$ with $\gamma\nabla f$, and view ARock as simply asynchronous block gradient descent.

\end{rem}

\subsection{The ARock algorithm}
ARock is an asynchronous-parallel, block, fixed-point algorithm, in which a shared solution vector $x^k$ is updated by a collection of $p$ computing nodes.

Take a space  $\HH$  on which to solve an optimization problem. $\HH$  can be the real space $\RR^N$ or a separable Hilbert space. Break this space into $m$ orthogonal subspaces: $\HH = \HH_{1} \times\ldots\times\HH_{m}$ so that vectors $x\in\HH$ can be written as $(x_1,x_2,\ldots,x_m)$  where each $x_i$  is $x$'s component in subspace $\HH_i$. Take an operator $T:\HH\to\HH$ that is $r$-Lipschitz for $0<r<1$. Let $S=I-T$ and $Sx=\p{S_{1}x,\ldots,S_{m}x}$ where $S_{j}x$ denotes the $j$'th block of $Sx$.

\begin{rem}[Conventions]
\emph{Superscripts} will denote the iteration
number of a sequence of points $x^0,x^1,x^2,\ldots$. \emph{Subscripts} will denote different blocks of a vector or operator, e.g., $x=(x_1,x_2,\ldots,x_m)$ and $Sx=\p{S_{1}x,\ldots,S_{m}x}$. {For instance, $x^k_l$ is the $l$th block of the $k$th iterate ($x^k$). $S_l x^k$ is the $l$th block of $S(x^k)$.}

\end{rem}

\begin{defn}[The ARock Algorithm]\label{def:ARock}
Let $\eta^{k}\in\RR$ be a series of step lengths and $i(k)\in\{1,\ldots,m\}$ be a series of block indices. Let $T$ be a nonexpansive operator with at least one fixed point $x^*$, and $S=I-T$. Take a starting point $x^0\in\HH$. Then the ARock algorithm \cite{PengXuYanYin2016_arock} is defined
via the iteration:

\begin{equation}\label{eq:def:ARock}
\mbox{for}~i=1,\ldots,m,\quad x_i^{k+1}  \gets
\begin{cases}
x_i^{k}-\eta^{k}S_{i}(\hx^{k}),& i=i\p{k},\\
x_i^{k},& i\neq i\p{k},
\end{cases}
\end{equation}
where the \textbf{delayed iterate} $\hx^{k}$ represents a possibly outdated version of the iteration vector $x^k$ used to make an update and the
\textbf{block index sequence} $i(k)$ specifies which block of $x^k$ is being updated to produce the next iterate $x^{k+1}$.

\end{defn}

The ARock algorithm resembles a block KM iteration. However we use a delayed iterate $\hx^k$ because of asynchronicity. In \Cref{ssec:Setup} we precisely define the block sequence $i(k)$, and the delayed iterate $\hx^k$. 

\subsection{Setup} \label{ssec:Setup}
We work with a probability measure space: $(\Omega,\Sigma,\PP)$\footnote{$\Omega$ is the probability space, $\Sigma$ is a sigma algebra, and $\PP$ is a corresponding probability measure.}. We now describe the delayed iterates (which is part of our model of asynchronicity) and the block index.

\subsubsection{Delayed iterates} \label{sec:Delayed-iterates}
Let $\vec{j} = (j_1,\ldots,j_m) \in \NN^m$ be a vector, and $x^0,x^1,x^2,\ldots$ a series of iterates. To model outdated solution data, we find it convenient to define:
\begin{align}
x^{k-\vec{j}} & = \p{x_1^{k-j_{1}},x_2^{k-j_{2}},\ldots,x_m^{k-j_{m}}}.
\end{align}
Hence we define a series of \textbf{delay vectors} $\vec{j}(0), \vec{j}(1),\vec{j}(2),\ldots$ in $\NN^m$, corresponding to $x^0, x^1, x^2,\ldots$ respectively. The components of the delay vector $\vec{j}\p{k} = \p{j\p{k,1},j\p{k,2},\ldots,j\p{k,m}}$ represent the staleness of the components of the solution vector $x^k$.
%
\begin{defn}[Delayed iterate]
The delayed iterate $\hx^k$ is defined as\footnote{\textbf{Stronger asynchronicity:} It is possible to have more general asynchronicity, where different components of the \textit{same block}, $x_l\in\HH_l$, have different ages. This leads to similar results, and a similar proof, but the current setup was chosen for simplicity.}:
\begin{align}
\hx^k &= x^{k-\vec{j}\p{k}},~\mbox{or equivalently,}\\
\hx^k= \p{\hx_{1}^{k},\hx_{2}^{k},\ldots,\hx_{m}^{k}} &= \p{x_{1}^{k-j\p{k,1}},x_{2}^{k-j\p{k,2}},\ldots,x_{m}^{k-j\p{k,m}}}.
\end{align}
\end{defn}
%
Lastly, we define the current delay $j(k)$ as follows\footnote{Notice the lack of a vector symbol: This distinguishes the current delay from the delay vector.}:
\begin{align}\label{eq:def:current-delay}
j\p{k} &= \max_{i}\cp{ j\p{k,i}}.
\end{align}
These delay vectors 
depend on the model of asynchronicity chosen. We consider two possibilities in this paper: stochastic and deterministic delays.

\subsubsection{Block sequence}\label{sec:Block-sequence}
We define following filtrations to represent the information that is accumulated as the algorithm runs.
\begin{align}
\cF^{k} & = \sigma\p{x^{0},x^{1},\ldots,x^{k},\vec{j}\p{0},\vec{j}\p{1},\ldots,\vec{j}\p{k}}\\
\cG^{k} & = \sigma\p{x^{0},x^{1},\ldots,x^{k}}
\end{align}
Here $\sigma(a,b,c,\ldots)$ represents the sigma algebra generated by $a,b,c,\ldots$.

\begin{asmp}[IID block sequence] \label{asmp:Block-sequence}
The sequence in which blocks of the solution vector are updated, $i(k)$, is a series of uniform\footnote{Nonuniform probabilities are a simple extension. However for simplicity we assume a uniform distribution.} IID random variables that takes values $1,2,\ldots,m$ each with probability $1/m$. $i(k)$ is independent of $\cF^k$. That is, $i(k)$ is independent of the sequence of iterates $(x^0,x^1,\ldots,x^k)$ and the sequence of delays $(\vec{j}\p{0},\vec{j}\p{1},\ldots,\vec{j}\p{k})$ jointly\footnote{Clearly this makes $i(k)$ independent of $\cG^k$ as well.}.
\end{asmp}
It is very difficult to remove the assumption that the block sequence is independent of the sequence of delays. Only a few papers that we are aware of make progress in eliminating this assumption \cite{SunHannahYin2017_asynchronous,LeblondPedregosaLacoste-Julien2017_asaga,CannelliFacchineiKungurtsevScutari2017_asynchronous}. However obtaining good convergence rates remains elusive.

Also removing the assumption of a random block sequence, and assuming, say, a cyclic choice as in \cite{SunYe2016_worstcase,ChowWuYin2017_cyclic} leads to at least an $m$-times slowdown of the algorithm in the worst case for smooth minimization \cite{SunYe2016_worstcase}. The block sequence will be IID if we allow all nodes to randomly update any block chosen in a uniform IID fashion, and computing each block is of equal difficulty. Future work may involve finding an intermediate scenarios between IID and cyclic block choices that still results in adequate rates.

\subsection{Lyapunov functions} \label{ssec:Lyapunov-functions}
In this framework, it is easy to generate an example where we have conditional expectation bound:
\begin{align}
\EE\sp{\n{x^{k+1}-x^*}^2 \big| \sigma\p{x^{0},x^{1},\ldots,x^{k},\vec{j}\p{0},\vec{j}\p{1},\ldots,\vec{j}\p{k}}} > \n{x^k-x^*}^2
\end{align}
for any nonzero step size. However usually some kind of monotonicity is necessary to prove convergence, especially linear convergence. In \cite{HannahYin2016_unbounded}, following on from \cite{PengXuYanYin2016_arock}, the authors propose an \textbf{asynchronicity error} term to add to the classical error:
\begin{align}\label{eq:def:Xi}
\underbrace{\xi^{k}}_{\text{Total error}} &= \underbrace{\n{ x^{k}-x^{*}} ^{2}}_{\text{Classical error}}+\underbrace{\frac{1}{m}\sum_{i=1}^{\infty}c_{i}\n{ x^{k+1-i}-x^{k-i}} ^{2}}_{\text{Asynchronicity error}}
\end{align}
for positive decreasing coefficients $(c_1,c_2,\ldots)$. Using carefully chosen coefficients and step size, they are able to prove convergence of ARock under unbounded delay for $T$ with Lipschitz constant $r=1$. This Lyapunov function appears naturally in the proof, and much like a well chosen basis in linear algebra, it seems to be the most natural error to consider when proving convergence. Refer to \cite{HannahYin2016_unbounded} for further motivation, and the general strategy for generating useful Lyapunov functions.

We use the same Lyapunov function in this paper to derive the main results. However our choice of coefficients is drastically different. Choosing the coefficients that yield strong results is a very involved process 
and is part of the technical innovation of this paper.

\subsection{Structure of the paper}
In \Cref{sec:New-results}, we systematically work through the points of the main argument in \Cref{ssec:Main-argument}, which bolsters our main thesis that asynchronous algorithms are drastically faster at scale, and discuss related work in \Cref{ssec:Related-work}. This culminates in our most important contribution: \Cref{thm:Linear-convergence-stochastic-delays-intro}, which essentially proves that asynchronous ARock has the same iteration complexity as its synchronous counterpart. This result is proven in \Cref{sec:Proof-for-Stochastic-Delays}, and \Cref{sec:Proof-for-Deterministic-Delays} introduces and proves a similar result for deterministic unbounded delays.

\section{New results} \label{sec:New-results}
In this section, we present our main results, that justify the proposition in our main argument in \Cref{ssec:Main-argument}. First we describe the implementation setup in \Cref{ssec:Implementation-setup}, that is, the kind of context where our main argument applies. Next in \Cref{ssec:Synchronization-penalty}, we describe some factors that cause synchronous algorithms to perform fewer epoch in a given time period. For instance, we prove that as the number of nodes increases, under a reasonable model synchronous algorithms suffer a $\Theta(\ln(p))$ slowdown, whereas asynchronous algorithms suffer no such penalty (more iterations). In \Cref{ssec:Convergence-rate-synchronous}, we derive a sharp convergence rate for synchronous-parallel ARock, so that we can compare iteration complexities. This appears to be a new result, with many implications in parallel optimization. In \Cref{ssec:Main-results}, we present our main results: The iteration complexity of (asynchronous) ARock is essentially the same as its synchronous counterpart. This result is our main theoretical contribution. The other subsections are independent contributions whose role is to justify our argument and bolster the important of the main theorems. Finally in \Cref{ssec:Related-work}, we discuss related work and compare results.

\subsection{Implementation setup} \label{ssec:Implementation-setup}
We now describe the implementation setup that we will be considering. The results and analysis may be more general than this setup, but being concrete about the setup allows us to compare synchronous vs. asynchronous implementations of the same algorithm. Our implementation setup will consist of:
\begin{enumerate}
\item \textbf{Central Parameter Server:} This is a central node that will maintain the solution vector $x$, and apply updates that are supplied to it\footnote{We could have considered a shared memory architecture, however we chose this setting since we are more focused on algorithm performance at scale.}: $x\leftarrow x - \eta S_i \hat{x}$.
\item \textbf{Computing Nodes:} There are $p$ nodes that read the solution vector $\hat{x}$ into local memory, calculate $S_i \hat{x}$, and send this update back to the central server.
\end{enumerate}

We now compare and contrast how a synchronous and asynchronous version of ARock would function:

\begin{itemize}
\item \textbf{Synchronous-parallel:} All computing nodes read the same solution vector $x$ from the server. All computing nodes will compute an update $S_i x$, where $i$ depends on the node, and send this update to the server. The server will save all received updates and, only after $x$ has been read by every node, apply the received updates to the solution vector. After this, all computing nodes will then read the same updated solution vector again, and the process repeats.
\item \textbf{Asynchronous-parallel:} All computing nodes will continually read the central solution vector $x$ into their local memory. This yields $\hat{x}$, a potentially inconsistently read solution vector. Each node will then calculate an update $S_i \hat{x}$ and send it back to the server. The server will apply any updates to the solution vector as they arrive.
\end{itemize}

Notice that in the synchronous implementation, a single slow node can hold the entire system up. In the asynchronous implementation, all nodes function independently, and never wait for any other node. The central server does not wait to receive an update from every node, but simply applies them as they come. We now analyze the synchronization penalty of these two implementation setups.

\subsection{Synchronization penalty} \label{ssec:Synchronization-penalty}
In this section we discuess point 2 of the main argument in \Cref{ssec:Main-argument}. We consider a simple and well-justified model of our implementation setup to investigate synchronization penalty. Even under perfect load balancing, this model will imply a slowdown for synchronous algorithms that increases as $\Theta(\ln(p))$, where $p$ is the number of processors. This may be even worse in a real implementation, where other factors may further disadvantage synchronous algorithms (such as the overhead associated with message passing and read/write locks). For the corresponding asynchronous algorithm, we show these problems do not occur. Hence asynchronous algorithms may compute far more iterations per second. 

Following \cite{SerpedinChaudhari2009_synchronization} (pp. 43) and adding modifications, we model the time taken for node $l$'s update as follows:
\begin{align}
P_l = R_l + C_l(i_l,m) + S_l.
\end{align}
Here $i_l$ is the block of the solution vector that node $l$ updates, and $C_l(i,m)$ represents the ``predictable'' portion of the update time (it is merely a function, not a random variable). This includes computation time, and the delay because of the limited bandwidth of the network. $C_l(i,m)$ is a function of $i$ because different blocks may have different sizes and difficulties. $R_l$ is the random delay involved in receiving the solution vector from the central server, and $S_l$ is the random delay involved in sending an update back to the server. $R_l$ and $S_l$ are assumed IID with exponential distribution of mean $\lambda_l$. This exponential model for the random portion of the delay has extensive theoretical and empirical justifications (see \cite{SerpedinChaudhari2009_synchronization}, pp. 44-45 for a discussion of the evidence).

\subsubsection{Synchronous algorithms and random delays}
Let's first consider the effect of random delays on the synchronization penalty. For simplicity, assume that $C_l(i,m)$ is constant over $i$ and $l$, and hence we can write this function as $C(m)$. Also assume we have $\lambda_l=\lambda$ for all $l$. This situation would occur if all blocks were of equal difficulty to update, and all nodes had the same computational power and network delay distribution. This is the ideal scenario, and yet we will observe a growing synchronization penalty with scale.

Because all nodes must finish updating for the next iteration to start, the iteration time $P$ is given by:
\begin{align}
P=C(m)+\max_{l=1,2,\ldots,p}\cp{R_l+S_l}.
\end{align}
Hence we have (using \cite{Eisenberg2008_expectation}):
\begin{align*}
\EE P-C(m) & \geq \EE \p{\max_{l=1,\ldots,p}\{R_{l}\}} =\lambda\sum_{l=1}^{p}\frac{1}{l} \geq \lambda\ln\p p.
\end{align*}

Let's now look at the time $\mathcal{T}(K)$ required for $K$ epochs, which corresponds to $Km/p$ iterations. Let\footnote{We write $A\sim B$ for random variables $A$ and $B$ if these variables have the same distribution.} $P^1,P^2,\ldots \sim P$. Then:
\begin{align*}
\mathcal{T}(K) &= \sum_{k=1}^{\lceil Km/p\rceil} P^k,\\
\EE \mathcal{T}(K) & \geq (Km/p) \EE P\\
&\geq (Km/p)\p{C(m)+\lambda\ln\p{p}}.
\end{align*}
Hence for small values of $p$, the expected time to reach $K$ epochs will decrease linearly with the number of nodes $p$. However as $p$ becomes larger, there is at least a $\Theta(\ln(p))$ penalty in how long this will take compared to a linear speedup.
%
\subsubsection{Asynchronous algorithms}
Using the same model, we now show that asynchronous algorithms have no such $\Theta(\ln(p))$ scaling penalty. The time taken for node $l$ to complete $k$ iterations is given by:
\begin{align}
S^k_l &= \sum_{j=1}^k P_l^j
\end{align}
where $P_l^k\sim P$. This is actually a \textbf{renewal process} with \textbf{interarrival time} $P_l$ (see \cite{MitovOmey2014_renewal,KellaStadje2006_superposition}). However if you consider the total number of iterations completed by all nodes together, then the time of the $k$'th iteration $S^k$ is known as a \textbf{superposition of renewal processes}. From \cite{KellaStadje2006_superposition} 1.4, as $k\to\infty$ we have:
\begin{align}
\frac{\EE S^k}{k} &\to \frac{\EE P}{p},\\
\text{(by convergence in the previus step) }\EE S^k &= k\frac{\p{C(m)+2\lambda}}{p}(1+o_{m,\lambda,p}\p{1}).
\end{align}
%
\begin{rem}[Notation]
The subscripts in $o_{m,\lambda,p}\p{1}$ denote that this term converges to $0$ as $k\to\infty$ in a way that depends on $m$, $p$, and $\lambda$.

\end{rem}

Hence the expected time to complete $K$ epochs is given by:
\begin{align}
\EE \mathcal{T}(K) &= \frac{Km}{p}\p{C(m)+2\lambda}(1+o_{m,\lambda,p}\p{1})
\end{align}
as $K\to\infty$. Hence it can be seen that asynchronous algorithms do not have a $\ln(p)$ penalty as $p$ becomes larger, when $K$ is sufficiently large. Hence for large $K$, asynchronous algorithms will compute at least $\Theta(\ln\p p)$ more epochs per second than synchronous algorithms.

\subsubsection{Heterogeneity and synchronous algorithms}
Sometimes a parallel problem \textit{cannot} be split into $m$ blocks in a way that updating each block is of equal difficulty (as was previously assumed in this subsection). This can cause significant synchronization penalty in the synchronous case, but has no such effect on asynchronous algorithms because computing nodes do not have to wait for slower nodes or blocks to complete.

Let us assume for the moment that there is no random component of the update time for a single node, and that all nodes have the same computational power. This means that the update time for node $l$ at iteration $k$ is simply:
\begin{align}
P^k_l &= C(i_l,m)
\end{align}
where $i_l$ is the block that node $l$ updates at iteration $k$. Assume also that at every iteration, each node $l$ will chose a random block to update, and hence $i_l$ is a uniform random variable on $\cp{1,2,\ldots,m}$. For the synchronous algorithm, we have an update time:
\begin{align}
P &= \max_{l=1,2,\ldots,p}\cp{C(i_l,m)}
\end{align}
Clearly then, as $p$ increases, we have:
\begin{align}
\EE P &\to \max_i C(i,m)
\end{align}
That is, the update time is determined by the most difficult block to update. Hence the the expected time for $K$ epochs is:
\begin{align}
\EE \mathcal{T}(K) &= \frac{Km}{p}\p{\max_i C(i,m) + o_{m}(1)}
\end{align}
as $p\to\infty$.

\subsubsection{Heterogeneity and asynchronous algorithms}
Now consider an asynchronous algorithm. The update time of a single node is:
\begin{align}
\EE P &= \EE C(i_l,m) = \frac{1}{m}\sum_{i=1}^{m}C(i,m)
\end{align}
Yet again we have a superposition of renewal processes, and hence from \cite{KellaStadje2006_superposition}, we have as $k\to\infty$:
\begin{align}
\frac{\EE S^{k}}{k} & \to\frac{\EE P}{p}\\
\EE S^{k} & =\frac{k}{p}\p{\frac{1}{m}\sum_{i=1}^{m}C(i,m)}\p{1+o_{m,p}\p 1}
\end{align}
Hence the expected time for $K$ epochs is given by:
\begin{align}
\EE\mathcal{T}\p K & =\frac{Km}{p}\p{\frac{1}{m}\sum_{i=1}^{m}C(i,m)}\p{1+o_{m,p}\p 1}
\end{align}
as $K\to\infty$. Notice that the time taken for an asynchronous algorithm is determined by the average difficulty of updating a block. Compare this to synchronous algorithms where the most difficult block determines the time complexity. If the difficulty of blocks is highly heterogeneous, asynchronous algorithms may complete far more iterations per second, even without considering network effects.

\subsubsection{Additional factors}
In the previous, we merely gave an analysis of a couple factors that decrease the number of epochs that synchronous solvers complete. Another factors is heterogeneous computing power of the computing nodes themselves, which causes disadvantages even if all blocks are the same difficulty. Also there is significant overhead associated with enforcing synchronization, as well as read and write locks.

\subsection{Iteration complexity for synchronous Block KM} \label{ssec:Convergence-rate-synchronous}
In this subsection, we start to look at point 2 of the main argument in \Cref{ssec:Main-argument}. In order to prove there is no iteration complexity penalty, we need to obtain tight rates of convergence for synchronous ARock (which is merely a synchronous block KM iteration). At every step, each of the $p$ processing nodes are given a random block to update with a KM-style iteration. Hence we have:
\begin{align}
x^{k+1} & =x^{k}-\eta^{k}P^{k}Sx^{k}\label{eq:def:KM-random-subset}
\end{align}
Here $P^{k}$ is a projection onto a random subset of $\cp{1,2,\ldots,m}$ of size $p$ (we assume $p\leq m$). We note that each block has a $\frac{p}{m}$ probability
of being updated on a given iteration.
%
\begin{defn}[Convergence rate]
An algorithm is said to linearly converge if the error $E(k)=\cO\p{R^k}$ for $0<R<1$. $R$ is called the \textbf{convergence rate}.

\end{defn}

\begin{defn}[Epoch iteration complexity]
The epoch iteration complexity $I(\eps)$ is the number of \textbf{epochs} required to decrease the error below $\eps E(0)$, where $E(0)$ is the initial error.
\end{defn}

This error could be the distance from the solution $\n{x^k-x^*}^2$ or the gap between the function value and its optimal value $f(x^k)-f^*$.

\begin{prop}[Convergence rate of block KM iterations] \label{prop:Block-KM-rates}
Let $T$ be an $r$-Lipschitz operator for $0<r<1$. Consider the random subset KM iteration defined in \Cref{eq:def:KM-random-subset} for $1\leq p\leq m$. A step size of $\eta^{k}=1$
optimizes the convergence rate. When this optimal step size is chosen,
we have convergence rate:
\begin{align*}
R & =1-\frac{p}{m}\p{1-r^{2}}
\end{align*}
and corresponding epoch iteration complexity:
\begin{align}
I(\eps) & =\p{\frac{1}{1-r^{2}}-\theta\frac{p}{m}}\ln\p{1/\eps}\label{eq:Optimal-rate-random-subset-KM}
\end{align}
for some $\theta\in\sp{\frac{1}{2},1}$.

\end{prop}

For problems of interest, the first term ($1/(1-r^2)$) will dominate the second. We are interested in huge-scale problems, which will usually have $m\gg p$ or $r\approx 1$.

This is proven in \Cref{ssec:Block-KM-Iterations}. Hence if we have either $r\to 1$ or $p/m\to 0$, then:
\begin{align}
I &= \p{1+o(1)}\frac{1}{1-r^2}\ln\p{1/\eps}
\end{align}
We will eventually prove that ARock has essentially the same iteration complexity.

This result allows us to obtain a sharp convergence rate and epoch iteration complexity for synchronous-parallel block gradient descent (of which block gradient descent is a special case of $p=1$). This appears to be a new result that extends recent work in \cite{TaylorHendrickxGlineur2017_exact} to the case of random-subset block gradient descent (which corresponds to the special case $m=1$, $p=1$).

\begin{cor}[Sharp Convergence Rate of Synchronous-Parallel Block Gradient Descent]\label{cor:BCD-synchronous-rates}
Let $f$ be $\mu$-strongly convex function with $L$-Lipcshitz gradient $\nabla f(x)$. Let $\kappa=L/\mu$ be the condition number. The operator $T=I-\frac{2}{\mu+L}\nabla f$ is $r=\p{1-\frac{2}{\kappa+1}}$ Lipschitz. The corresponding block KM iteration \Cref{eq:def:KM-random-subset} with optimal step size $\eta^k=1$ is equivalent to synchronous-parallel block gradient descent with step size $2/\p{\mu+L}$. The linear convergence rate $R$ and epoch iteration complexity $I(\eps)$ with respect to the error $\n{x^k-x^*}^2$ are given by the following:
\begin{alignat}{3}
R & =1-4\frac{p}{m}\frac{\kappa}{\p{\kappa+1}^{2}} && =1-4\frac{p}{m\kappa}\p{1+\cO\p{1/\kappa}}\label{eq:Convergence-rate-gradient-sharp}\\
I\p{\eps} & =\frac{1}{4}\p{\kappa+\cO\p 1}\ln\p{1/\eps}&&\label{eq:Iteration-complexity-gradient-sharp}
\end{alignat}
as $\kappa\to\infty$. Lastly, this convergence rate is sharp.
\end{cor}

This is proven in \Cref{ssec:Block-gradient-descent}. Among other things, our main results will show that asynchronous-parallel block-gradient descent has epoch iteration complexity that is asymptotically equal to $\frac{1}{4}\kappa \ln \p{1/\eps}$, which is the complexity of synchronous-parallel block coordinate descent.

\begin{rem}[Composite objectives]
Let $g(x)=\sum_{i=1}^ng\p{x_i}$ be separable, with each component convex, and subdifferentiable. The exact same convergence rate and complexity clearly holds for block proximal gradient descent. This is because the corresponding nonexpansive operator $T=\p{I+\partial g}^{-1}\circ\p{I-\frac{2}{\mu+L}\nabla f}$ is also $1-\frac{2}{\kappa+1}$-Lipschitz, and hence all preceding theory applies.

\end{rem}

\subsection{Iteration complexity for ARock} \label{ssec:Main-results}
In this subsection we present our theoretical results for ARock, which completes part 2 of the main argument in \Cref{ssec:Main-argument}. That is, ARock (and hence all its special cases) has essentially the same iteration complexity as its synchronous counterpart. However we present the results in simplified form so as to more effectively communication the main message. Full versions of these results that contain more technical details are given and proven in the proof sections.

\subsubsection{Stochastic delays}
The first result is convergence under a stochastic unbounded delays from a fixed distribution (though this can be weakened to a changing distribution).

\begin{asmp}[Stochastic unbounded delays] \label{asmp:Stochastic-delays}
The sequence of delay vectors $\vec{j}(0), \vec{j}(1), \vec{j}(2),\ldots$ is IID, and $\vj(k)$ is independent of $\cG^k$. That is, $\vj(k)$ is indepdendent of the iterates $\p{x^0,x^1,x^2,\ldots,x^k}$.
\end{asmp}

We can imagine a large optimization problem being solved on a busy network. Nodes are continually sending and receiving their updates. Traffic is chaotic and there is some kind of distribution of how long the information take to get from one node to the rest.

The delay vectors $\vj(k)$ have a fixed distribution (though this can be relaxed). Define
\begin{align}
P_l=\PP\sp{j(k)\geq l}
\end{align}
We let $\rho$ be defined by:
\begin{align}
\rho &= 1-\frac{1}{m}\p{1-r^2} \label{eq:def:rho}
\end{align}
which is the linear convergence rate of the corresponding synchronous KM algorithm (\Cref{eq:def:KM-random-subset}) with $p=1$, and ideal step size. We also define probability moments:
\begin{align}
M_1 &= \sum_{l=1}^\infty P_l\rho^{-l/2},\quad&
M_2 &= \sum_{l=1}^\infty P_l^{1/2}\rho^{-l/2}\label{eq:def:M}
\end{align}
and step size:
\begin{align}
\eta^k &=\eta_1 \triangleq \p{1+m^{-1/2}\p{\p{1-r^2}^{1/2}M_1+2M_2}} \label{eq:def:eta1-final}
\end{align}
Notice that these moments are a function of $m$. This is immediately clear because $\rho$ is a function of $m$. Also the probability distribution of the delays may depend on $m$ in a way that depends on the network and how you decide to scale up the computation to a higher number of nodes or blocks.

\begin{thm}[Linear convergence for stochastic delays] \label{thm:Linear-convergence-stochastic-delays-intro}
Let \Cref{asmp:Block-sequence} and \Cref{asmp:Stochastic-delays} hold. Let $M_1$, and $M_2$ be finite and $\cO\p{m^{q}}$ for $0\leq q < 1/2$. Let $\eta^k=\eta_1$. Then there exist positive coefficients $(c_i)_{i=1}^\infty$ of the Lyapunov function $\xi^k$ such that we have the following linear convergence rate and iteration complexity respectively:
\begin{align}
\EE\sp{\xi^{k+1}\big|\cG^k} & \leq \underbrace{\Big(1-\frac{1}{m}\p{1-r^{2}}}_{\text{Ideal synchronous rate}}+\underbrace{\cO\p{m^{q-3/2}}\Big)}_{\text{Asynchronicity penalty}}\xi^k\label{eq:Asymptotic-rate-stochastic},\\
I\p{\eps} & =\p{1+\cO\p{m^{-1/2+q}}}\p{\frac{1}{1-r^{2}}}\ln\p{1/\eps},
\end{align}
as $m\to\infty$.

\end{thm}

The full version of this is \Cref{thm:Linear-convergence-stochastic-delays-full}, which is proven in \Cref{ssec:Proof-of-theorem-stochastic} after a series of results built up in \Cref{sec:Proof-for-Stochastic-Delays}.

Comparing this to \Cref{prop:Block-KM-rates}, we can see that as $m\to\infty$, the iteration complexities of asynchronous and synchronous ARock approach the same value. Thus at scale there is no iteration complexity penalty for using outdated information, so long as that information does not become too old as you scale up the problem (i.e. $M_1$ and $M_2$ don't grow too fast as $m$ increases.).

\begin{rem}[Reduction to synchronous case]
We will see that when there is no asynchronicity, the Lyapunov function will reduce to the classical error $\xi^k=\n{x^k-x^*}^2$, and $\eta_1=1$. Also the convergence rate and iteration complexity will exactly equal the values obtained in \Cref{prop:Block-KM-rates} (again for $p=1$).

\end{rem}

\subsubsection{Deterministic delays}
We also prove a similar convergence result for deterministic unbounded delays. However we present the section in \Cref{sec:Proof-for-Deterministic-Delays}, because the specifics of the theorem are a little subtle. The proof of the stochastic delay result can be seen as a warmup to the deterministic result.

\subsubsection{Completing the main argument}
Hence our convergence rate results complete point 2 of the main argument. We have \textit{more epochs} per given time period, and the epochs make the same progress because iterations are of the \textit{same quality}. Hence asynchronous algorithms may drastically outperform synchronous ones in this setting.

\subsection{Related Work} \label{ssec:Related-work}
Though asynchronous algorithms were invented long ago, there has been a lot of recent interest, especially for random-block-coordinate-type algorithms, such as ``Hogwild!'' \cite{RechtReWrightNiu2011_hogwild}. In \cite{AvronDruinskyGupta2014}, the authors prove linear convergence for an asynchronous stochastic linear solver. In \cite{LiuWrightReBittorfSridhar2015_asynchronous}, the authors prove function-value convergence for asynchronous stochastic coordinate descent. They prove $\cO\p{1/k}$ convergence for $f$ convex with $\nabla f$ Lipschitz, and linear convergence when $f$ is also strongly convex. This was extended in \cite{LiuWright2015_asynchronous} to composite objective functions.

For condition number $\kappa$, they report a per-iteration linear convergence rate of
\begin{align}
1-\frac{1}{2m\kappa}
\end{align}
This implies an iteration complexity approximately $8$ times higher than our result (where our result matches asymptotically to the complexity of the corresponding synchronous algorithm). For asynchronicity to be useful, the reduced penalty would need to compensate for this $8$-fold increase in iterations. Like almost all recent work except for \cite{HannahYin2016_unbounded,PengXuYanYin2016_convergence}, they assume a bounded delay $\tau$. For linear speedup, they require $\tau=\cO(m^{1/2})$ and $\tau=\cO(m^{1/4})$ for composite objectives. We do not require bounded delays for linear speedup, only sufficiently slowly growing moments of delay. For bounded delay our condition is $\tau=\cO(m^{q})$ for $0\leq q<\frac{1}{2}$ for composite and non-composite objectives.

Our work is also more general, since we use the operator setting. So not only do the main results apply to gradient descent, and proximal gradient, but any other algorithm that can be written in a block fixed-point form.

In \cite{ManiaPanPapailiopoulosRechtRamchandranJordan2015_perturbed}, authors achieve a linear speedup but with a far higher iteration complexity of $\Theta\p{\kappa^2\ln\p{1/\eps}}$. When the iteration complexity is increased by a factor of $\kappa$, the linear speedup may be of limited utility. We also note that our conditions for linear speedup are weaker than theirs (which is $\tau=\cO\p{m^{1/6}}$), and they need to assume that $x^k$ remains bounded, which is unjustified.

In \cite{PengXuYanYin2016_convergence} prove function-value linear convergence of an asynchronous block proximal gradient algorithm under unbounded delays. However it is unclear how the iteration complexity they obtain compares to the corresponding synchronous algorithm. Also our result applies to a KM iteration, of which block proximal gradient is a special case.

In \cite{LianZhangHsiehHuangLiu2016_comprehensive}, the authors review a number of asynchronous algorithm analyses and collect conditions necessary for linear speedup on a fixed problem. In light of potentially increased complexity, as seen in \cite{ManiaPanPapailiopoulosRechtRamchandranJordan2015_perturbed}, a linear speedup does not imply that asynchronous algorithms will run faster in time. It only implies that the potential slowdown factor for using asynchronicity is bounded for a given problem, but this bound may be as large as $\kappa$. What we prove in this paper is much stronger, that the slowdown factor (in terms of iterations) for using asynchronicity is asymptotically $1$, i.e. that the slowdown is negligible.

\subsection{Unbounded delays}
Almost all work on asynchronous algorithms except for \cite{HannahYin2016_unbounded,PengXuYanYin2016_convergence} assumed bounded delays. That is, there was a limit $\tau$ such that $j(k)\leq\tau$ for all $k$. However there may be no bound on the delay in practice: There is always the possibility of an arbitrarily large network delay. Also this $\tau$ needs to be known in advance in order to set the step size, which is impractical. A large $\tau$ will hinder the convergence rate by limiting the step size, even if a delay of $\tau$ is extremely unlikely. In this work we assume no bound on the delay. In \Cref{thm:Linear-convergence-stochastic-delays-intro}, we are able to obtain fast convergence if the delay distribution is not too spread out. In \Cref{thm:Linear-convergence-deterministic-delays-full} we are able to determine a convergence rate that depends only on the current delay conditions, not on the worst case behavior of delay on the entire time that the algorithm runs.

\section{Analysis of ARock under Stochastic Delays} \label{sec:Proof-for-Stochastic-Delays}
We now give the full version of \Cref{thm:Linear-convergence-stochastic-delays-intro}. We let the coefficients in $\xi^k$ be given by:
\begin{align}\label{eq:thm:Coefficients-stochastic}
c_i &= m^{1/2}\p{1+\p{1-r^2}^{1/2}}\sum^\infty_{l=i} P_l^{-1/2}\rho^{-(l/2-i+1)}
\end{align}
and define the constant:
\begin{align}
\eta_2 &= m^{1/2}\p{1-r^2}^{-1/2}M_1^{-1}\label{eq:def:eta2-final}
\end{align}
Finally, define the following \textbf{convergence rate function}:
\begin{align}
R\p{\eta,\gamma}& =\p{1-\frac{\eta}{m}\p{1-r^{2}}\p{1-\eta/\gamma}}\label{eq:def:R}
\end{align}
Note that we have $R<1$ when $0<\eta<\gamma$, and the rate is optimized when $\eta=(1/2)\gamma$. Also $\rho\leq R$ for $\eta\leq 1$, where $\rho = 1-(1/m)(1-r^2)$ is the optimal rate that we wish to prove $R$ is close to.

\begin{thm}[Linear convergence for stochastic delays] \label{thm:Linear-convergence-stochastic-delays-full}
Let \Cref{asmp:Block-sequence} and \Cref{asmp:Stochastic-delays} hold. Let the step size $\eta^k$ be $\cF^k$-measurable, and satisfy $\eta^k\leq\eta_1$. Let the probability moments $M_1$ and $M_2$ defined in \Cref{eq:def:M} be finite. Consider the Lyapunov function defined in \Cref{eq:def:Xi} with constants given by \Cref{eq:thm:Coefficients-stochastic}. Then we have the following linear convergence rate:
\begin{align}
\EE\sp{\xi^{k+1}\big|\cG^{k}} &\leq R\p{\eta^k,\eta_2} \xi^k\label{eq:thm:Linear-convergence-stochastic}
\end{align}
Additionally, let $\eta^k=\eta_1$ and assume that $M_1, M_2 = \cO\p{m^{q}}$ for $0\leq q < 1/2$. As the number of blocks $m$ approaches $\infty$, we have the following linear convergence rate and iteration complexity respectively:
\begin{align}
R\p{\eta_{1},\eta_{2}} & =\underbrace{\p{1-\frac{1}{m}\p{1-r^{2}}}}_{\text{Ideal synchronous rate}}+\underbrace{\frac{2}{m^{3/2}}\p{1-r^{2}}^{3/2}\p{\p{1-r^{2}}^{1/2}M_{1}+M_{2}}}_{\text{Asynchronous rate penalty}}\p{1+o\p 1}\label{eq:Asymptotic-rate-stochastic-full}\\
I\p{\eps} & =\p{1+\underbrace{2m^{-1/2}\p{\p{1-r^{2}}^{1/2}M_{1}+M_{2}}}_{\text{Highest order penalty term}}+o\p{m^{-1/2+q}}}\p{\frac{1}{1-r^{2}}}\ln\p{1/\eps}\label{eq:Asymptotical-complexity-full}
\end{align}

\end{thm}

We chose to include the exact form of the highest order iteration complexity penalty. This allows us to calculate approximately how large $m$ has to be in order for asynchronicity to cause negligible penalty.

We now prove \Cref{thm:Linear-convergence-stochastic-delays-full} in a way that emphasizes the reasons and intuition behind our approach -- especially the strategic way in which the coefficients are chosen.

\subsection{Preliminary results}
Let $x^*$ be any solution, and set $x^*=0$ with no loss in generality, to make the notation more compact. This can be achieved by translating the origin of the coordinate system to $x^*$. Hence $\n{ x^{k}}$ is the distance from the solution. The starting point of our analysis is the following\footnote{We will use an abuse of notation in this paper. We equate $S_i(x)\in\HH_i$ (the components of $S(x)$ in the $i$th block) and $(0,\ldots,0,S_i(x),0,\ldots,0)\in\HH_1\times\ldots\times\HH_m$ (the projection of $S(x)$ to the $i$'th subspace). Hence we can write the ARock iteration more compactly as $x^{k+1}=x^{k}-\eta^{k}S_{i\p{k}}\hx^{k}$.}:
\begin{align*}
\EE\sp{\n{ x^{k+1}} ^{2}|\cF^{k}} & = \EE\sp{\n{ x^{k}-\eta^{k}S_{i\p{k}}\hx^{k}} ^{2}|\cF^{k}}\\
 & = \n{ x^{k}} ^{2}+\EE\sp{-2\eta^{k}\dotp{ x^{k},S_{i\p{k}}\hx^{k}} +\p{\eta^{k}}^{2}\n{ S_{i\p{k}}\hx^{k}} ^{2}|\cF^{k}}.
\end{align*}
Here the expectation is taken over only the block index $i\p{k}$ (Recall \Cref{asmp:Block-sequence}). We let the step size $\eta^k$ be $\cG^k$-measurable (and hence $\cF^k$-measurable). Essentially this means that the step size $\eta^k$ can depend only on the sequence $\p{x^0,x^1,\ldots,x^k}$, and not the block index or delay. Hence
\begin{align}\label{eq:Starting-point}
\EE\sp{\n{ x^{k+1}} ^{2}|\cF^{k}} & = \n{ x^{k}} ^{2}\underbrace{-2\frac{\eta^{k}}{m}\dotp{ x^{k},S\hx^{k}} }_{\text{cross term}}+\frac{\p{\eta^{k}}^{2}}{m}\n{ S\hx^{k}}^{2}.
\end{align}

We now present a simple lemma on the operator $S$ that will be used in the convergence proof (This is the operator version of Theorem 2.1.12 in \cite{Nesterov2013_introductory}).

\begin{lem}\label{lem:S-inequality-Nesterov}
Let $S=I-T$, where $T$ is an $r$-Lipschitz operator. Then for all
$x,y\in\HH$ we have:
\begin{align}
\dotp{Sy-Sx,y-x} & \geq\frac{1}{2}\n{Sy-Sx}^{2}+\frac{1}{2}\p{1-r^{2}}\n{y-x}^{2}
\end{align}

\end{lem}

\begin{proof}
\begin{align*}
\text{(\ensuremath{T} is \ensuremath{r}-Lipschitz) }r^{2}\n{y-x}^{2} & \geq\n{Ty-Tx}^{2}\\
 & =\n{\p{I-S}y-\p{I-S}x}^{2}\\
 & =\n{Sy-Sx}^{2}-2\dotp{Sy-Sx,y-x}+\n{y-x}^{2}\\
\text{(rearrange) }\dotp{Sy-Sx,y-x} & \geq\frac{1}{2}\n{Sy-Sx}^{2}+\frac{1}{2}\p{1-r^{2}}\n{y-x}^{2} \qedhere
\end{align*}
\end{proof}

\subsection{The cross term}
\begin{rem}[Strategy] \label{rem:Strategy}
Consider \Cref{eq:Starting-point} again. The $\n{S\hx^{k}}^{2}$ term in \Cref{eq:Starting-point} can be thought of as a ``waste'' term that has to be negated. In light of \Cref{lem:S-inequality-Nesterov}, a $-\dotp{S\hx^k,\hx^k}$ term can be used to generate a $-\n{S\hx^k}^2$ term to clean this waste. In addition to cleaning this waste, ideally we would have a $-\n{x^k}^2$ to help prove linear convergence, but instead \Cref{lem:S-inequality-Nesterov} produces $-\n{\hx^k}^2$.

The strategy we pursue is as follows: The cross term $-\dotp{S\hx^k,x^k}$ is approximately equal to $-\dotp{S\hx^k,\hx^k}$, which allows us to clean  the $\n{S\hx^k}^2$ term. The $-\n{\hx^k}^2$ that is also generated is approximately equal to $-\n{x^k}^2$, which helps prove linear convergence. However, there is an error associated with this ``conversion''. Finally, this conversion error is negated by the use of as Lyapunov function (see \cref{eq:def:Xi}).

\end{rem}

\Cref{lem:AM-GM-trick} will eventually allow us to quantify the error associated with converting $-\dotp{S\hx^k,x^k}$ to $-\dotp{S\hx^k,\hx^k}$, and the error associated with converting $-\n{\hx^k}^2$ to $-\n{x^k}^2$ mentioned in \Cref{rem:Strategy}.

\begin{lem}\label{lem:AM-GM-trick}
Let $a>0$, $j(k)$ be the current delay, $\eta^k$ be the current step size, and $\eps_1,\eps_2,\ldots>0$ be a series of parameters. Then we have:
\begin{align}
a\n{x^k-\hx^k} &\leq \frac{1}{2}a^{2}\eta^{k}\p{\sum_{i=1}^{j(k)}\frac{1}{\eps_{i}}}+\frac{1}{2}\frac{1}{\eta^{k}}\sum_{i=1}^{j(k)}\p{\eps_{i}\n{x^{k+1-i}-x^{k-i}}^{2}}\label{eq:AM-GM-trick}
\end{align}
\end{lem}

\begin{proof}
See \cite{HannahYin2016_unbounded,PengXuYanYin2016_arock}.
\end{proof}

\begin{rem}[Free parameters] \label{rem:Free-parameters}
\Cref{lem:E-Norm-Fk} generates some positive parameters $\eps_{1},\eps_{2},\ldots>0$ and $\del_{1},\del_{2},\ldots>0$. It's not immediately clear what these parameters should be set to. However we will see in \Cref{ssec:Linear-convergence-stochastic} if they are properly chosen, they can be used to construct a Lyapunov function that will allow us to prove linear convergence.

\end{rem}

We will make use of \Cref{lem:AM-GM-trick} twice with parameter sets $\seq{\eps}$ and $\seq{\del}$ respectively. To simplify notation, we define:
\begin{align}
E_{j} &= \sum_{i=1}^{j}\frac{1}{\eps_{i}} & D_j &= \sum_{i=1}^{j}\frac{1}{\del_{i}} \label{eq:def:E-and-D}
\end{align}

\begin{lem}\label{lem:E-Norm-Fk}
Let \Cref{asmp:Block-sequence} hold. Let $\eps_{1},\eps_{2},\ldots>0$ and $\del_{1},\del_{2},\ldots>0$
be a sequence of parameters, and $E_j$, $D_j$ defined as above. Let $\eta^k$ be $\cG^k$-measurable. ARock yields the following inequality:
\begin{align*}
\EE\sp{\n{x^{k+1}}^2\big|\cF^k} & \leq\p{1-\frac{\eta^{k}}{m}\p{1-r^{2}}\p{1-\eta^{k}D_{j(k)}}}\n{x^{k}}^{2}+\frac{1}{m}\sum_{i=1}^{j(k)}\p{\del_{i}\p{1-r^{2}}+\eps_{i}}\n{x^{k+1-i}-x^{k-i}}^{2}\\
 & -\frac{\eta^{k}}{m}\n{S\hx^{k}}^{2}\p{1-\eta^{k}\p{1+E_{j(k)}}}
\end{align*}
\end{lem}

\begin{proof}
We make use of \Cref{lem:AM-GM-trick} twice in this proof, with parameter sets $\seq{\eps}$ and $\seq{\del}$ respectively.
\begin{align}
& -2\frac{\eta^{k}}{m}\dotp{x^{k},S\hx^{k}}\nonumber\\
& =-2\frac{\eta^{k}}{m}\dotp{\hx^{k},S\hx^{k}}-2\frac{\eta^{k}}{m}\dotp{x^{k}-\hx^{k},S\hx^{k}}\nonumber\\
& \leq-\frac{\eta^{k}}{m}\p{\n{S\hx^{k}}^{2}+\p{1-r^{2}}\n{\hx}^{2}}+2\frac{\eta^{k}}{m}\n{x^{k}-\hx^{k}}\cdot\n{S\hx^{k}}\nonumber\\
& \leq-\frac{\eta^{k}}{m}\p{\n{S\hx^{k}}^{2}+\p{1-r^{2}}\n{\hx^{k}}^{2}}+2\frac{\eta^{k}}{m}\p{\frac{1}{2}\n{S\hx^{k}}^{2}\eta^{k}E_{j(k)}+\frac{1}{2}\frac{1}{\eta^{k}}\sum_{i=1}^{j(k)}\p{\eps_{i}\n{x^{k+1-i}-x^{k-i}}^{2}}}\nonumber\\
& = -\frac{\eta^{k}}{m}\p{1-r^{2}}\n{\hx^{k}}^{2}+\frac{1}{m}\sum_{i=1}^{j(k)}\eps_{i}\n{x^{k+1-i}-x^{k-i}}^{2}-\frac{\eta^{k}}{m}\n{S\hx^{k}}^{2}\p{1-\eta^{k}D_{j(k)}}\label{eq:Fundamental-lemma-middle}
\end{align}
For sufficiently small step size, this inequality allows us to negate the $\n{S\hx^{k}}^{2}$ terms. Now let's examine $-\n{\hx^{k}}^{2}$, which we convert to a $-\n{x^{k}}^{2}$ term (and some error) for linear convergence.
\begin{align*}
-\n{\hx^{k}}^{2} & =-\n{x^{k}}^{2}-2\dotp{\hx^{k}-x^{k},x^{k}}-\n{x^{k}-\hx^{k}}^{2}\\
& \leq-\n{x^{k}}^{2}+2\n{\hx^{k}-x^{k}}\n{x^{k}}\\
\text{(\Cref{lem:AM-GM-trick})}& \leq-\n{x^{k}}^{2}+\n{x^{k}}^{2}\eta^{k}D_{j(k)}+\frac{1}{\eta^{k}}\sum_{i=1}^{j(k)}\p{\del_{i}\n{x^{k+1-i}-x^{k-i}}^{2}}\\
& =-\p{1-\eta^{k}D_{j(k)}}\n{x^{k}}^{2}+\frac{1}{\eta^{k}}\sum_{i=1}^{j(k)}\p{\del_{i}\n{x^{k+1-i}-x^{k-i}}^{2}}
\end{align*}
Hence substituting into \eqref{eq:Fundamental-lemma-middle}, we have
\begin{align*}
-2\frac{\eta^{k}}{m}\dotp{x^{k},S\hx^{k}} & \leq-\frac{\eta^{k}}{m}\p{1-r^{2}}\p{1-\eta^{k}D_{j(k)}}\n{x^{k}}^{2}+\frac{\eta^{k}}{m}\p{1-r^{2}}\frac{1}{\eta^{k}}\sum_{i=1}^{j(k)}\p{\del_{i}\n{x^{k+1-i}-x^{k-i}}^{2}}\\
& +\frac{1}{m}\sum_{i=1}^{j(k)}\eps_{i}\n{x^{k+1-i}-x^{k-i}}^{2}-\frac{\eta^{k}}{m}\n{S\hx^{k}}^{2}\p{1-\eta^{k}D_{j(k)}}\\
& =-\frac{\eta^{k}}{m}\p{1-r^{2}}\p{1-\eta^{k}D_{j(k)}}\n{x^{k}}^{2}+\frac{1}{m}\sum_{i=1}^{j(k)}\p{\del_{i}\p{1-r^{2}}+\eps_{i}}\n{x^{k+1-i}-x^{k-i}}^{2}\\
& -\frac{\eta^{k}}{m}\n{S\hx^{k}}^{2}\p{1-\eta^{k}D_{j(k)}}
\end{align*}
Using \eqref{eq:Starting-point} immediately yields the result.
\end{proof}

\subsection{The Lyapunov function}
We now consider how the Lyapunov function defined in \Cref{eq:def:Xi} changes in size from step to step. The reason that a Lyapunov function is needed is to deal with the $\n{x^{k+1-i}-x^{k-i}}^2$ terms. They cannot be negated like $\n{S\hx^k}^2$ terms, and so must be incorporated into the error by using a Lyapunov function. Let $P_l = \PP\sp{j\p{k}\geq l}$.

\begin{lem}\label{lem:Expectation-xi-Gk-stochastic}
Let the conditions of \Cref{lem:E-Norm-Fk} and \Cref{asmp:Stochastic-delays} hold. Define
\begin{align}
\eta_{1} & =\p{1+\frac{c_{1}}{m}+\n{\frac{1}{\eps_{i}}}_{\ell^{1}}}^{-1}\label{eq:def:eta1-general}\\
\eta_{2} & =\p{\sum_{i=1}^{\infty}\frac{P_{i}}{\del_{i}}}^{-1}\label{eq:def:eta2-general}\\
R\p{\eta,\gamma}& =\p{1-\frac{\eta}{m}\p{1-r^{2}}\p{1-\eta/\gamma}}\label{eq:def:R2}
\end{align}
Let $\eta^k$ be $\cG^k$-measurable, and  $\eta^k\leq\eta_1$. Then ARock satisfies:
\begin{align*}
\EE\sp{\xi^{k+1}\big|\cG^{k}} & \leq\n{x^{k}}^{2}R\p{\eta^k,\eta_2} +\frac{1}{m}\sum_{i=1}^{\infty}\p{\p{\eps_{i}+\p{1-r^{2}}\delta_{i}}P_{i}+c_{i+1}}\n{x^{k+1-i}-x^{k-i}}^{2}
\end{align*}

\end{lem}

Notice that we have defined $\eta_1$ and $\eta_2$ in terms of the unspecified parameters $\seq{\eps}$ and $\seq{\del}$. Eventually, we will set $\eps_i=m^{1/2}P_i^{-1/2}\rho^{i/2}$ and $\del_i = m^{1/2}(1-r^2)^{-1/2}\rho^{i/2}$ for reasons that will be explained in \Cref{ssec:Proof-of-theorem-stochastic}. With this parameter choice, the definitions of $\eta_1$ and $\eta_2$ will match \cref{eq:def:eta1-final} and \cref{eq:def:eta2-final} respectively.

\begin{proof}
%
\begin{align}
\EE\sp{\xi^{k+1}\big|\cF^{k}}&=\underbrace{\EE\sp{\n{x^{k+1}}^{2}\big|\cF^{k}}}_{A}+\underbrace{\frac{c_{1}}{m}\EE\sp{\n{x^{k+1}-x^{k}}^{2}\big|\cF^{k}}}_{B}+\underbrace{\frac{1}{m}\sum_{i=1}^{\infty}c_{i+1}\n{x^{k+1-i}-x^{k-i}}^{2}}_{C}
\end{align}
We obtain a bound on $A$ from \Cref{lem:E-Norm-Fk}. $B$ follows by the definition of ARock:
\begin{align*}
B & =\frac{c_{1}}{m}\frac{\p{\eta^{k}}^{2}}{m}\n{S\hx^{k}}^{2}.
\end{align*}
$C$ contains no expectation because it is $\cF^k$ measurable. Hence we have:
\begin{align}
\begin{aligned}
\EE\sp{\xi^{k+1}\big|\cF^{k}} &\leq\n{x^{k}}^{2}\p{1-\frac{\eta^{k}}{m}\p{1-r^{2}}\p{1-\eta^{k}\p{D_{j(k)}}}}
-\frac{\eta^{k}}{m}\n{S\hx^{k}}^{2}\p{1-\eta^{k}\p{1+\frac{c_{1}}{m}+E_{j(k)}}}\\
&+\frac{1}{m}\sum_{i=1}^{j(k)}\p{\eps_{i}+\p{1-r^{2}}\delta_{i}}\n{x^{k+1-i}-x^{k-i}}^{2}+\frac{1}{m}\sum_{i=1}^{\infty}c_{i+1}\n{x^{k+1-i}-x^{k-i}}^{2}
\end{aligned}\label{eq:Branch-point}
\end{align}
Notice that $E_{j}=\sum_{i=1}^{j}1/\eps_{i}\leq\n{1/\eps_{i}}_{\ell^{1}}$ for all $j$, and that therefore the step size condition eliminates the $\n{S\hx^{k}}^{2}$
term.

Now it becomes necessary to take expectations over the delay distribution (by taking the expectation with respect to $\cG^{k}$ instead of $\cF^k$). Notice that for a sequence $\p{\gamma_{1},\gamma_{2},\ldots}$, we have: $\EE\sp{\sum_{i=1}^{j(k)}\gamma_{i}\big|\cG^{k}}=\sum_{i=1}^{\infty}P_{i}\gamma_{i}$. This yields:
\begin{align*}
\EE\sp{\xi^{k+1}\big|G^{k}} & \leq\n{x^{k}}^{2}\p{1-\frac{\eta^{k}}{m}\p{1-r^{2}}\p{1-\eta^{k}\p{\sum_{i=1}^{\infty}\frac{P_{i}}{\del_{i}}}}}\\
 & +\frac{1}{m}\sum_{i=1}^{\infty}\p{\eps_{i}+\p{1-r^{2}}\delta_{i}}P_{i}\n{x^{k+1-i}-x^{k-i}}^{2}+\frac{1}{m}\sum_{i=1}^{\infty}c_{i+1}\n{x^{k+1-i}-x^{k-i}}^{2}
\end{align*}
which completes the proof.
\end{proof}

\subsection{Linear convergence} \label{ssec:Linear-convergence-stochastic}
The right-hand side in \Cref{lem:Expectation-xi-Gk-stochastic} closely resembles $\xi^k$. Ideally, we have:
\begin{align}
\EE\sp{\xi^{k+1}\big|\cG^{k}} &\leq \gamma \xi^k
\end{align}
for $0<\gamma<1$, and some choice of parameters $(\eps_1,\eps_2,\ldots)$, $(\del_1,\del_2,\ldots)$ and coefficients $(c_1,c_2,\ldots)$. In this section we will derive such a result by carefully chosing these parameters. However, in order to derive a coefficient formula, we need the following lemma.

\begin{lem}[Coefficient formula] \label{lem:Coefficient-formula}
Let $0<\rho<1$ and let $(s_1,s_2,\ldots)$ be a positive sequence. Consider the coefficient formula:
\begin{align}
c_i &= \sum^\infty_{l=i}s_l \rho^{-(l-i+1)}.
\end{align}
If $c_1<\infty$, then we have $c_i\downarrow 0$ and:
\begin{align}
\rho c_i &= c_{i+1} + s_i
\end{align}
\end{lem}

\begin{proof}
\begin{align*}
\rho c_{i} & =\sum_{l=i}^{\infty}s_{l}\rho^{-\p{l-i}} =\sum_{l=i+1}^{\infty}s_{l}\rho^{-\p{l-\p{i+1}+1}}+s_{i} =c_{i+1}+s_{i}
\end{align*}
Clearly this implies $c_i\downarrow 0$, since $\rho<1$ and coefficients are nonnegative.
\end{proof}

Recall that $\rho$ is defined in \cref{eq:def:rho}.

\begin{prop}[Linear convergence for stochastic delays] \label{pro:Linear-convergence-general-stochastic}
Let \Cref{asmp:Block-sequence} hold. Let $\eta^k\leq \eta_1$, and let $\eps_{1},\eps_{2},\ldots>0$ and $\del_{1},\del_{2},\ldots>0$
be a sequence of parameters. Let
\begin{align}
\sum^\infty_{l=i}\p{\eps_l+\p{1-r^2}\delta_l }P_l\rho^{-l} < \infty\label{eq:General-moment-finite-stochastic}
\end{align}

With the choice of coefficients\footnote{This formula will eventually match \cref{eq:thm:Coefficients-stochastic} when the parameters $\eps_{1},\eps_{2},\ldots>0$ and $\del_{1},\del_{2},\ldots>0$ are chosen later.}:
\begin{align}
c_i &= \sum^\infty_{l=i}\p{\eps_l+\p{1-r^2}\delta_l }P_l\rho^{-(l-i+1)}\label{eq:Coefficient-formula}
\end{align}
We have the following linear convergence result:
\begin{align}
\EE\sp{\xi^{k+1}\big|\cG^{k}} &\leq R\p{\eta^k,\eta_2} \xi^k
\end{align}

\end{prop}

\begin{proof}
By applying \Cref{lem:Coefficient-formula} with $s_l=P_l(\eps_{l}+\p{1-r^{2}}\delta_{l})$, we obtain:
\begin{align*}
\rho c_{i} &=\p{\eps_{i}+\p{1-r^{2}}\delta_{i}}P_{i}+c_{i+1}
\end{align*}
Hence from \Cref{lem:Expectation-xi-Gk-stochastic}, we have:
\begin{align*}
\EE\sp{\xi^{k+1}\big|\cG^{k}} & \leq\n{x^{k}}^{2}R\p{\eta^k,\eta_2}+\frac{1}{m}\sum_{i=1}^{\infty}\p{\p{\eps_{i}+\p{1-r^{2}}\delta_{i}}P_{i}+c_{i+1}}\n{x^{k+1-i}-x^{k-i}}^{2}\\
& \leq\n{x^{k}}^{2}R\p{\eta^k,\eta_2}+\frac{1}{m}\sum_{i=1}^{\infty}\rho c_{i}\n{x^{k+1-i}-x^{k-i}}^{2}
 \leq\max\p{\rho,R\p{\eta^k,\eta_2}}\xi^{k} = \rho\xi^k\qedhere
\end{align*}
The last line follows, because $0\leq\eta^k\leq\eta_1\leq1$ implies $\rho\leq R(\eta^k,\eta_2)$.
\end{proof}

\subsection{Proof of \Cref{thm:Linear-convergence-stochastic-delays-full}} \label{ssec:Proof-of-theorem-stochastic}
Recall we have the following step size restriction $\eta^k\leq \eta_1$ (with $\eta_1$ defined in \cref{eq:def:eta1-general}) coupled with the convergence rate:
\begin{align*}
R\p{\eta^k,\eta_2}&=\p{1-\frac{\eta^k}{m}\p{1-r^{2}}\p{1-\eta^k/\eta_{2}}}\text{, for } \eta_{2} =\p{\sum_{i=1}^{\infty}\frac{P_{i}}{\del_{i}}}^{-1}
\end{align*}

We now prove \Cref{thm:Linear-convergence-stochastic-delays-full}. However we do so in a way that justifies the choice of parameters that we use. In the case of $\seq{\eps}$ there is a best choice. For other parameters, we simply pick a sensible though not necessarily optimal choice. First though, a couple of lemmas are needed to look at the asymptotic convergence rate and iteration complexity.

\begin{lem}\label{lem:Asymptotic-step-size-convergence-rate}
Say that as $m\to\infty$, we have $x(m)=\cO(1)$, $y(m)=\cO(1)$, and:
\begin{align*}
t_{1} & =1+x(m)m^{-a}+o\p{m^{-a}},\,1>a>0\\
t_{2}^{-1} & =y(m)m^{-b}+o\p{m^{-b}},\,1>b>0
\end{align*}
Then
\begin{align*}
R\p{t_{1},t_{2}} & =1-\frac{1}{m}\p{1-r^{2}}\p{1-\p{x(m)m^{-a}+y(m)m^{-b}}+o\p{m^{-a}+m^{-b}}}
\end{align*}
\end{lem}

\begin{proof}
\begin{align*}
R\p{t_{1},t_{2}} & =1-\frac{1}{m}\p{1-r^{2}}t_{1}\p{1-t_{1}t_{2}^{-1}}\\
t_{1}\p{1-t_{1}t_{2}^{-1}} & =\p{1-xm^{-a}+o\p{m^{-a}}}\p{1-\p{1-xm^{-a}+o\p{m^{-a}}}\p{ym^{-b}+o\p{m^{-b}}}}\\
 & =\p{1-xm^{-a}+o\p{m^{-a}}}\p{1-\p{1+\cO\p{m^{-a}}}\p{ym^{-b}+o\p{m^{-b}}}}\\
 & =\p{1-xm^{-a}+o\p{m^{-a}}}\p{1-ym^{-b}+o\p{m^{-b}}}\\
 & =1-xm^{-a}-ym^{-b}+o\p{m^{-a}+m^{-b}}\\
R\p{t_{1},t_{2}} & =1-\frac{1}{m}\p{1-r^{2}}\p{1-\p{xm^{-a}+ym^{-b}}+o\p{m^{-a}+m^{-b}}}\qedhere
\end{align*}
\end{proof}

\begin{lem} \label{lem:Asymptotic-iteration-complexity}
Let $x(m)=\cO(1)$, and $0\leq a$. If the linear convergence rate $R$ satisfies:
\begin{align}
R=1-\frac{1}{m}\p{1-r^{2}}\p{1-(xm^{-a}+o\p{m^{-a}})}
\end{align}
as $m\to\infty$. Then we have:
\begin{align}
I\p \eps &= \p{1+xm^{-a}+o\p{m^{-a}}}\p{\frac{1}{1-r^{2}}}\ln\p{1/\eps}
\end{align}

\end{lem}

\begin{proof}
%
\begin{align*}
\eps &=\p{R\p{\eta_{1},\eta_{2}}}^{I(\eps)m}\\
\ln\p{1/\eps} & =-I(\eps)m\ln\p{R\p{\eta_{1},\eta_{2}}}\\
 & =I(\eps)m\p{\frac{1}{m}\p{1-r^{2}}\p{1-xm^{-a}+o\p{m^{-a}}}+\cO\p{\p{\frac{1}{m}\p{1-r^{2}}\p{1-xm^{-a}+o\p{m^{-a}}}}^{2}}}\\
 & =I(\eps)m\p{\frac{1}{m}\p{1-r^{2}}\p{1-xm^{-a}+o\p{m^{-a}}}+\cO\p{\p{\frac{1}{m}\p{1-r^{2}}\cO\p 1}^{2}}}\\
 & =I(\eps)m\p{\frac{1}{m}\p{1-r^{2}}\p{1-xm^{-a}+o\p{m^{-a}}}}\\
I(\eps) & =\p{1+xm^{-a}+o\p{m^{-a}}}\p{\frac{1}{1-r^{2}}}\ln\p{1/\eps}\qedhere
\end{align*}
\end{proof}

\begin{proof}[Proof of \Cref{thm:Linear-convergence-stochastic-delays-full}]
We start with the conditions of \Cref{pro:Linear-convergence-general-stochastic}. $M_1$ and $M_2$ being finite corresponds to \Cref{eq:General-moment-finite-stochastic} in \Cref{pro:Linear-convergence-general-stochastic}.

It is immediately possible to maximize $\eta_{1}$ over the sequence $\eps_{i}$, by letting $\eps_{i}=\sqrt{m}P_{i}^{-1/2}\rho^{i/2}$. All thing being equal, increasing $\eta_{1}$ allows for a better convergence rate by increasing the range of possible step sizes. This leads to:
\begin{align}
\eta_{1} & =\p{1+m^{-1}\p{1-r^{2}}\sum_{l=1}^{\infty}\del_{l}P_{l}\rho^{-l}+2m^{-1/2}\sum_{l=1}^{\infty}P_{l}^{1/2}\rho^{-l/2}}^{-1}
\end{align}
and leaves $\eta_{2}$ unchanged. We also let $\del_{l}=d\rho^{l/2}$,
with $d$ to be determined later. This yields:
\begin{align*}
\eta_{1} & =\p{1+m^{-1}\p{1-r^{2}}dM_{1}+2m^{-1/2}M_{2}}^{-1}, &\eta_{2}^{-1} & =d^{-1}M_1
\end{align*}
for $M_{1} =\sum_{l=1}^{\infty}P_{l}\rho^{-l/2}$, and  $M_{2}  =\sum_{l=1}^{\infty}P_{l}^{1/2}\rho^{-l/2}$. Now we set $d=am^{1/2}\p{1-r^{2}}^{-1/2}$, for $a$ to be determined later. We make this choice so that that $\eta_1$ and $\eta_2$ are both $1+\cO\p{m^{q-1/2}}$ for large $m$, which optimizes the asymptotic rate at which $\eta_1$ converges to $1$. Recall that the moments $M_1$ and $M_2$ vary with $m$, and satisfy $M_1, M_2 = \cO\p{m^q}$ for $0\leq q < 1/2$. This yields:
\begin{align*}
\eta_{1} & =\p{1+am^{-1/2}\p{1-r^{2}}^{1/2}M_{1}+2m^{-1/2}M_{2}}^{-1}\\
 & =1-m^{-1/2}\p{a\p{1-r^{2}}^{1/2}M_{1}+2M_{2}}+\cO\p{m^{2q-1}}\\
\eta_{2} & =a^{-1}m^{-1/2}\p{1-r^{2}}^{1/2}M_{1}
\end{align*}
It is clear that $M_1$, $M_2$ and $c_i$ match the formulas given in the \Cref{thm:Linear-convergence-stochastic-delays-full} (see \cref{eq:def:M}, and \cref{eq:thm:Coefficients-stochastic}). We now use \Cref{lem:Asymptotic-step-size-convergence-rate} with $t_{1}=\eta_{1}$, $t_{2}=\eta_{2}$, $a=b=q-\frac{1}{2}$, etc.
\begin{align*}
R\p{\eta_{1},\eta_{2}} & =1-\frac{1}{m}\p{1-r^{2}}\p{1-m^{-1/2}\p{\p{a+a^{-1}}\p{1-r^{2}}^{1/2}M_{1}+2M_{2}}+o\p{m^{-1/2+q}}}
\end{align*}
Clearly to minimize the lowest order asynchronicity penalty, we should let $a=1$. With this choice, we have proven \Cref{eq:thm:Linear-convergence-stochastic}. Then using this in conjunction with \Cref{lem:Asymptotic-iteration-complexity} completes the proof of \Cref{thm:Linear-convergence-stochastic-delays-full}.
\end{proof}

\section{Deterministic Unbounded Delays} \label{sec:Proof-for-Deterministic-Delays}
We now present a convergence result for deterministic unbounded delays.

\begin{asmp}[Deterministic unbounded delays] \label{asmp:Deterministic-delays}
The sequence of delay vectors $\vec{j}(0), \vec{j}(1), \vec{j}(2),\ldots$ is an arbitrary sequence in $\NN^m$.
\end{asmp}

If there can be no assumption made on the distribution of the delays, we a have a slightly weaker result. Whereas before it was possible to use a constant step size, here the step size and convergence rate at step $k$ depend on the current delay $j\p{k}$. Results that assume bounded delay $\tau$ need to know $\tau$ in advance to set the step size. If $\tau$ is very large, this will decrease the allowable timestep, which will slow convergence -- even if a delay of $\tau$ is very rare. Here we make no such assumption and the step size and convergence rate are adaptive to whatever delay conditions exist in the system. The step size $\eta^k$ is a function of the current delay $j(k)$, that must be measured (or upper bounded) at each iteration $k$.

The delay is allowed to be unbounded. The larger the delay, the less progress is made, but \textit{some progress is always made at every step}. We define a \textbf{good behavior boundary}
\begin{align}
T(m) &= bm^{q}+d
\end{align}
for $0\leq q<1/2$, for arbitrary\footnote{These parameters are arbitrary, but setting them involves a trade-off as we will later see. The larger $b$ is, the larger the asynchronicity penalty, and larger $m$ has to be to ensure negligible penalty as will be seen in \Cref{thm:Linear-convergence-deterministic-delays-full}.}  parameters $b,d>0$. Any delay less than this will not create a noticeable penalty in the progress of an algorithm in a large system (i.e. as $m\to\infty$) as compared to synchronous ARock. However if the delay is much larger than $T(m)$, the progress can eventually become vanishingly small. $T(m)$ can become very large since it grows as $\Theta\p{m^q}$.

We let $\rho$ again be defined as in \Cref{eq:def:rho}. For arbitrary\footnote{Again, $c$ can be set freely, but larger $c$ is, the more the convergence rate is penalized for exceeding $T(m)$. However the smaller $c$ is, the larger the asynchronicity penalty will be, and the large $m$ needs to be to ensure negligible penalty (much like $b$).} $c>0$, let
\begin{align}
\gamma&=\rho-cm^{-q}
\end{align}
We define the coefficients in the Lyapunov function from \Cref{eq:def:Xi}:
\begin{align}
c_i &= m^{1/2}\p{1+2\p{1-r^{2}}}\frac{\p{\gamma/\rho}^{i}}{1-\gamma/\rho}\label{eq:Coefficients-optimal-deterministic}
\end{align}
We also define the following step size functions:
\begin{align}
H_{1}\p j & =\left(1+c^{-1}m^{q-1/2}\p{3+\gamma^{-j}}\right)^{-1}, &&
H_{2}\p j =2cm^{\frac{1}{2}-q}\gamma^{j}\label{eq:def:H}
\end{align}
These functions play similar roles to $\eta_1$ and $\eta_2$ from \Cref{thm:Linear-convergence-stochastic-delays-full}. We again use the same convergence-rate function $R$ as defined in \Cref{eq:def:R}. Let $\eta^{k}$ be $\cF^{k}$ measurable.

\begin{thm} \label{thm:Linear-convergence-deterministic-delays-full}
Let \Cref{asmp:Block-sequence} and \Cref{asmp:Deterministic-delays} hold. Consider the Lyapunov function defined in \Cref{eq:def:Xi} with coefficients given by \Cref{eq:Coefficients-optimal-deterministic}. Let  $\eta^{k}\leq H_{1}\p j$. Then we have:
\begin{align}
\EE\sp{\xi^{k+1}\big|\cG^{k}} & \leq R\p{\eta^{k},H_{2}\p{j\p k}}\xi^{k}\label{eq:thm:Linear-convergence-deterministic-full}
\end{align}
Additionally, let $\eta^{k}=H_{1}\p{j\p k}$. If we have $j\p k\leq T\p m=bm^{q}+d$ for some $q\in[0,\frac{1}{2})$ and $b,d>0$, then the convergence rate satisfies:
\begin{align*}
R & \leq1-\frac{1}{m}\p{1-r^{2}}\p{1-\frac{1}{c}m^{q-1/2}\p{3+\frac{3}{2}\exp\p{bc}}+o\p{m^{q-\frac{1}{2}}}}
\end{align*}
which corresponds to iteration complexity:
\begin{align*}
I\p{\eps} & =\p{1+\frac{1}{c}m^{q-1/2}\p{3+\frac{3}{2}\exp\p{bc}}+o\p{m^{q-\frac{1}{2}}}}\p{\frac{1}{1-r^{2}}}\ln\p{1/\eps}
\end{align*}
\end{thm}

\begin{rem}
Much like in the stochastic delay case (see \Cref{rem:Free-parameters}), we prove a more general result where the Lyapunov function, and step size depend on a series of parameters $\eps_{1},\eps_{2},\ldots>0$
and $\del_{1},\del_{2},\ldots>0$. We make a sensible choice of these parameters to make the result more simple and interpretable.

\end{rem}

\subsection{Starting point}
The starting point is \Cref{eq:Branch-point}. However in this result, we assume that $\eta^k$ is actually $\cF^k$ measurable (that is, it is now allowed to also depend on the delays  $\p{\vj(0),\vj(1),\vj(2),\ldots,\vj(k)}$, whereas before it could only depend on the iterates $(x^0,x^1,x^2,\ldots)$). Recall the definitions of $E_j$ and $D_j$ given in \cref{eq:def:E-and-D}.

\begin{lem} \label{lem:Expectation-xi-Fk-deterministic}
Let \Cref{asmp:Block-sequence} and \Cref{asmp:Deterministic-delays} hold. Let $\eps_{1},\eps_{2},\ldots>0$
and $\del_{1},\del_{2},\ldots>0$ be a sequence of parameters. Let
$\eta^{k}$ be $\cF^{k}$-measurable. Define the step size functions:
\begin{align*}
h_{1}(j) & =\p{1+\frac{c_{1}}{m}+E_j}^{-1}, &&
h_{2}(j) =D^{-1}_j
\end{align*}
Let $\eta^{k}\leq h_{1}\p{j\p k}$. Then ARock yields the following inequality:
\begin{align*}
\EE\sp{\xi^{k+1}\big|\cF^{k}} & \leq\n{x^{k}}^{2}R\p{\eta^{k},h_{2}\p{j(k)}} +\frac{1}{m}\sum_{i=1}^{\infty}\p{\p{\eps_{i}+\p{1-r^{2}}\delta_{i}}+c_{i+1}}\n{x^{k+1-i}-x^{k-i}}^{2}
\end{align*}
\end{lem}

The step size function $h_1$ and $h_2$ are complex expressions. When we eventually set the free parameters  $\eps_{1},\eps_{2},\ldots>0$
and $\del_{1},\del_{2},\ldots>0$, we simplify these functions with inequalities to yield $H_1$ and $H_2$, which are much easier to interpret.

\begin{proof}
From \Cref{eq:Branch-point} we have:
\begin{align*}
\EE\sp{\xi^{k+1}\big|\cF^{k}} & \leq\n{x^{k}}^{2}\p{1-\frac{\eta^{k}}{m}\p{1-r^{2}}\p{1-\eta^{k}D_{j(k)}}}-\frac{\eta^{k}}{m}\n{S\hx^{k}}^{2}\p{1-\eta^{k}\p{1+\frac{c_{1}}{m}+E_{j(k)}}}\\
 & +\frac{1}{m}\sum_{i=1}^{j(k)}\p{\eps_{i}+\p{1-r^{2}}\delta_{i}}\n{x^{k+1-i}-x^{k-i}}^{2}+\frac{1}{m}\sum_{i=1}^{\infty}c_{i+1}\n{x^{k+1-i}-x^{k-i}}^{2}\\
 & \leq\n{x^{k}}^{2}R\p{\eta^{k},h_{2}\p{j(k)}}-\frac{\eta^{k}}{m}\n{S\hx^{k}}^{2}\p{1-\eta^{k}/h_{1}\p{j(k)}}\\
 & +\frac{1}{m}\sum_{i=1}^{\infty}\p{\p{\eps_{i}+\p{1-r^{2}}\delta_{i}}+c_{i+1}}\n{x^{k+1-i}-x^{k-i}}^{2}
\end{align*}
Setting $\eta^{k}\leq h_{1}\p{j\p k}$ eliminates the $\n{S\hx^{k}}^{2}$
term, and yields the desired result.
\end{proof}

\subsection{Linear convergence}\label{ssec:Proof-of-thm-2}

\begin{prop}[Linear convergence for deterministic delays]\label{pro:Linear-convergence-general-deterministic}
Let the conditions of \Cref{lem:Expectation-xi-Fk-deterministic} hold. Also assume:
\begin{align}
\sum_{i=1}^{\infty}\p{\eps_{i}+\p{1-r^{2}}\del_{i}}\rho^{-i}<\infty\label{eq:General-moment-finite-deterministic}
\end{align}
Let the coefficients of the Lyapunov function be given by:
\begin{align}
c_i &= \sum_{l=i}^\infty \p{\eps_i+\p{1-r^2}\del_i}\rho^{l-i+1}
\end{align}
Then we have the following linear convergence rate at step $k$:
\begin{align}
\EE\sp{\xi^{k+1}\big|\cF^{k}} & \leq R\p{\eta^{k},h_{2}\p{j(k)}}\xi^{k}
\end{align}

\end{prop}

\begin{proof}
We apply \Cref{lem:Coefficient-formula} to \Cref{lem:Expectation-xi-Fk-deterministic} with $s_i=\eps_i+\p{1-r^2}\del_i$, and we obtain:
\begin{align}
\eps_i+\p{1-r^2}\del_i + c_{i+1} &\leq \rho c_i
\end{align}
Hence we have:
\begin{align}
\EE\sp{\xi^{k+1}\big|\cF^{k}} & \leq\n{x^{k}}^{2}R\p{\eta^{k},h_{2}\p{j(k)}} +\frac{1}{m}\sum_{i=1}^{\infty}\p{\p{\eps_{i}+\p{1-r^{2}}\delta_{i}}+c_{i+1}}\n{x^{k+1-i}-x^{k-i}}^{2}\\
& \leq\n{x^{k}}^{2}R\p{\eta^{k},h_{2}\p{j(k)}}+\frac{1}{m}\sum_{i=1}^{\infty}\rho c_{i}\n{x^{k+1-i}-x^{k-i}}^{2}\\
& \leq\max\cp{R\p{\eta^{k},h_{2}\p{j(k)}},\rho}\xi^{k}\\
& =R\p{\eta^{k},h_{2}\p{j(k)}}\xi^{k}
\end{align}
The last line follows, because $0\leq\eta^k\leq h_1(j)\leq1$ implies $\rho\leq R(\eta^k,\eta_2)$.
\end{proof}

\subsection{Proof of \Cref{thm:Linear-convergence-deterministic-delays-full}} \label{ssec:Proof-of-thm-3}

\begin{proof}[Proof of \Cref{thm:Linear-convergence-deterministic-delays-full}]
For the first part of the theorem, we simply set
\begin{align*}
\eps_{i} & =m^{1/2}\gamma^{i}\\
\del_{i} & =2m^{1/2}\gamma^{i}
\end{align*}
This choice automatically satisfies that conditions of \Cref{pro:Linear-convergence-general-deterministic} (i.e. \Cref{eq:General-moment-finite-deterministic}), and the coefficient formula that arises matches \Cref{eq:Coefficients-optimal-deterministic}.

We have step size functions $h_{1}$ and $h_{2}$ that are given by complicated expressions. Notice that we may overestimate $h_{1}$ and underestimate $h_{2}$, and the linear convergence result still holds. Hence we will simplify these step size expressions to give a more concise step size rule. We have:
\begin{align*}
\sum_{i=1}^{j}\gamma^{-j} & \leq\gamma^{-j}\sum_{i=0}^{\infty}\gamma^{i} =\gamma^{-j}\frac{1}{1-\gamma} \leq c^{-1}m^{q}\gamma^{-j}
\end{align*}
which yields:
\begin{align*}
\p{h_{2}\p j}^{-1} & =\sum_{i=1}^{j}\frac{1}{\del_{i}} \leq\frac{1}{2c}m^{q-\frac{1}{2}}\gamma^{-j}=H_{2}\p j
\end{align*}
which matches \Cref{eq:def:H}. We also have:
\begin{align*}
c_{1} & =\sum_{l=1}^{\infty}\p{\eps_{l}+\p{1-r^{2}}\del_{l}}\rho^{-l} =m^{1/2}\p{1+2\p{1-r^{2}}}\sum_{l=1}^{\infty}\p{\gamma/\rho}^{l} =m^{1/2}\p{1+2\p{1-r^{2}}}\frac{\gamma}{\rho-\gamma} \leq3c^{-1}m^{1/2+q}
\end{align*}
And hence:
\begin{align*}
h_{1}(j) & =\left(1+\frac{c_{1}}{m}+\sum_{i=1}^{j}\frac{1}{\eps_{i}}\right)^{-1} \geq\left(1+\frac{3}{c}m^{q-1/2}+c^{-1}m^{q-1/2}\gamma^{-j}\right)^{-1} =\left(1+\frac{1}{c}m^{q-1/2}\p{3+\gamma^{-j}}\right)^{-1}=H_{1}\p j
\end{align*}
which matches \Cref{eq:def:H}. Hence  \Cref{eq:thm:Linear-convergence-deterministic-full} is proven.

For the second part of the theorem, we need asymptotic expressions for $H_{1}\p T$ and $H_{2}\p T$ to determine an asymptotic convergence rate and iteration complexity. We first bound $\gamma^{-T}$
\begin{align*}
\gamma^{-T} & =\p{1-\p{1-\gamma}}^{-T}
 =\exp\p{T\p{1-\gamma}+\cO\p{T\p{1-\gamma}^{2}}} =\exp\p{bc+\cO\p{m^{-q}}} =\exp\p{bc}+\cO\p{m^{-q}}
\end{align*}
Hence this yields:
\begin{align*}
\p{h_{2}\p T}^{-1} & \leq\frac{1}{2c}m^{q-\frac{1}{2}}\p{\exp\p{bc}+\cO\p{m^{-q}}}\\
h_{1}(T) & \geq\left(1+\frac{1}{c}m^{q-1/2}\p{3+\exp\p{bc}+\cO\p{m^{-q}}}\right)^{-1} =1-\frac{1}{c}m^{q-1/2}\p{3+\exp\p{bc}}+o\p{m^{q-\frac{1}{2}}}
\end{align*}
Therefore, if $j\p k\leq T$, we have convergence rate:
\begin{align*}
R\p{H_{1}\p{j\p k},h_{2}\p{j\p k}} & \leq R\p{H_{1}\p{j\p T},H_{2}\p{j\p T}}\\
 & =1-\frac{1}{m}\p{1-r^{2}}\p{1-\frac{1}{c}m^{q-\frac{1}{2}}\p{3+\frac{3}{2}\exp\p{bc}}+o\p{m^{q-\frac{1}{2}}}}\\
 & =1-\frac{1}{m}\p{1-r^{2}}\p{1-\cO\p{m^{q-\frac{1}{2}}}}
\end{align*}
by \Cref{lem:Asymptotic-step-size-convergence-rate}. By \Cref{lem:Asymptotic-iteration-complexity} this corresponds to iteration complexity:
\begin{align*}
I\p{\eps} & =\p{1+\frac{1}{c}m^{q-1/2}\p{3+\frac{3}{2}\exp\p{bc}}+o\p{m^{q-\frac{1}{2}}}}\p{\frac{1}{1-r^{2}}}\ln\p{1/\eps}\\
 & =\p{1+\cO\p{m^{q-\frac{1}{2}}}}\p{\frac{1}{1-r^{2}}}\ln\p{1/\eps}
\end{align*}
which completes the proof.
\end{proof}

\pdfbookmark[0]{References}{References}
\printbibliography
\appendix
\section{Auxillary Results}

\subsection{Block KM Iterations}\label{ssec:Block-KM-Iterations}

\begin{proof}[Proof of \Cref{prop:Block-KM-rates}]
Taking conditional expectation on
\begin{align*}
\n{x^{k+1}}^{2} & =\n{x^{k}}^{2}-2\eta^{k}\dotp{x^{k},P^{k}x^{k}}+\p{\eta^{k}}^{2}\n{P^{k}x^{k}}^{2}
\end{align*}
with respect to $i(k)$ yields
\begin{align}
\EE\sp{\n{x^{k+1}}^{2}\big|x^{k}} & =\n{x^{k}}^{2}-2\eta^{k}\frac{p}{m}\dotp{x^{k},Sx^{k}}+\p{\eta^{k}}^{2}\frac{p}{m}\n{Sx^{k}}^{2}\nonumber \\
\text{(By \Cref{lem:S-inequality-Nesterov})} & \leq\n{x^{k}}^{2}-\eta^{k}\frac{p}{m}\p{\n{Sx^{k}}^{2}+\p{1-r^{2}}\n{x^{k}}^{2}}+\p{\eta^{k}}^{2}\frac{p}{m}\n{Sx^{k}}^{2}\nonumber \\
 & =\p{1-\eta^{k}\frac{p}{m}\p{1-r^{2}}}\n{x^{k}}^{2}-\eta^{k}\frac{p}{m}\p{1-\eta^{k}}\n{Sx^{k}}^{2},\label{eq:Branch-point-block-KM-1}
\end{align}
Let $\eta^{k}\leq1$, and we follow on from \eqref{eq:Branch-point-block-KM-1}:
\begin{align*}
\text{(By \ensuremath{(1-r)}-strong monotonicity of \ensuremath{S})} & \leq\p{1-\eta^{k}\frac{p}{m}\p{1-r^{2}}}\n{x^{k}}^{2}-\p{1-r}^{2}\eta^{k}\frac{p}{m}\p{1-\eta^{k}}\n{x^{k}}^{2}\\
 & =\p{1-\frac{p}{m}+\frac{p}{m}\p{1-\eta^{k}\p{1-r}}^{2}}\n{x^{k}}^{2}.
\end{align*}
Now let $\eta^{k}\geq1$, and we follow on again from \ref{eq:Branch-point-block-KM-1}.
\begin{align*}
 & \EE\sp{\n{x^{k+1}}^{2}\big|x^{k}}\\
\text{(Since \ensuremath{S} is \ensuremath{(1+r)}-Lipschitz)} & \leq\p{1-\eta^{k}\frac{p}{m}\p{1-r^{2}}}\n{x^{k}}^{2}-\eta^{k}\frac{p}{m}\p{1-\eta^{k}}\p{1+r}^{2}\n{x^{k}}^{2}\\
 & =\p{1-\frac{p}{m}+\frac{p}{m}\p{1-\eta^{k}\p{1+r}}^{2}}\n{x^{k}}^{2}.
\end{align*}

It can be verified that every single inequality for $\eta^{k}\leq1$
is an equality for $T=rI$, and every single inequality for $\eta^{k}\geq1$
is an equality for $T=-rI$. Therefore the inequalities give a sharp
rate of convergence. This rate is optimized when $\eta^{k}=1$, and
matches the rate given in \Cref{eq:def:rho}.

Let's now look at the corresponding iteration complexity:
\begin{align*}
\p{1-\frac{p}{m}\p{1-r^{2}}}^{I(\eps)\frac{m}{p}} & =\eps\\
I(\eps) & =\frac{p}{m}\p{\frac{\ln\p{1/\eps}}{-\ln\p{1-\frac{p}{m}\p{1-r^{2}}}}}
\end{align*}
Note that:
\begin{align*}
1-x & \leq\frac{-x}{\ln\p{1-x}}\leq1-\frac{1}{2}x
\end{align*}
 and hence
\begin{align*}
 & I(\eps)=\frac{1}{1-r^{2}}\p{1-\theta\frac{p}{m}\p{1-r^{2}}}\ln\p{1/\eps}
\end{align*}
where $\theta\in\sp{\frac{1}{2},1}$. This matches \Cref{eq:Optimal-rate-random-subset-KM}, hence the proof is complete.
\end{proof}

\subsection{Random-Subset Block Gradient Descent} \label{ssec:Block-gradient-descent}

\begin{proof}[Proof of \Cref{cor:BCD-synchronous-rates}]
First we consider the operator $T=I-\frac{2}{L+\mu}\nabla f$:
\begin{align*}
\n{T(y)-T(x)}^{2}
 & =\n{y-x}^{2}-\frac{4}{\mu+L}\dotp{\nabla f\p y-\nabla f\p x,y-x}+\p{\frac{2}{L+\mu}}^{2}\n{\nabla f\p y-\nabla f\p x}^{2}\\
\text{(By Theorem 2.1.12 of \cite{Nesterov2013_introductory})} & \leq\n{y-x}^{2}-\frac{4}{\mu+L}\p{\frac{\mu L}{\mu+L}\n{x-y}^{2}+\frac{1}{\mu+L}\n{\nabla f\p y-\nabla f\p x}^{2}}\\
 & +\p{\frac{2}{L+\mu}}^{2}\n{\nabla f\p y-\nabla f\p x}^{2}\\
 & =\n{y-x}^{2}\p{1-\frac{4\mu L}{\p{\mu+L}^{2}}} =\n{y-x}^{2}\p{1-\frac{4\kappa}{\p{1+\kappa}^{2}}} =\n{y-x}^{2}\p{1-\frac{2}{\kappa+1}}^{2}
\end{align*}
Hence $T$ is $\p{1-\frac{2}{\kappa+1}}$-Lipschitz.

We now determine the linear convergence rate:
\begin{align*}
r & =1-\frac{2}{\kappa+1}\\
1-r^{2} & =\frac{4}{\kappa+1}\p{\frac{\kappa}{\kappa+1}} =\frac{4\kappa}{\p{\kappa+1}^{2}}
\end{align*}
And hence using \Cref{prop:Block-KM-rates} we have:
\begin{align*}
R & =1-\frac{p}{m}\p{1-r^{2}} =1-\frac{p}{m}\frac{4\kappa}{\p{\kappa+1}^{2}}
=1-4\frac{p}{m\kappa}\p{1+\cO\p{1/\kappa}}
\end{align*}
which matches \Cref{eq:Convergence-rate-gradient-sharp}. Next we determine the iteration complexity using \Cref{prop:Block-KM-rates}
again:
\begin{align*}
I(\eps) & =\p{\frac{1}{1-r^{2}}-\theta\frac{p}{m}}\ln\p{1/\eps} =\p{\frac{\p{\kappa+1}^{2}}{4\kappa}-\theta\frac{p}{m}}\ln\p{1/\eps} =\frac{1}{4}\p{\kappa+2+\frac{1}{\kappa}-4\theta\frac{p}{m}}\ln\p{1/\eps}\\ &=\frac{1}{4}\p{\kappa+\cO\p 1}\ln\p{1/\eps}
\end{align*}
This matches \Cref{eq:Iteration-complexity-gradient-sharp}.

Lastly, we prove that the convergence rate is sharp. Recall the proof of \Cref{prop:Block-KM-rates}. Now consider the function $f\p x=\frac{1}{2}\mu\n x^{2}$. This leads to $T=I\p{1-\frac{2}{\kappa+1}}$, which is equal to the worst-case example for $\eta\leq1$. Additionally, consider $f\p x=\frac{1}{2}L\n x^{2}$, which leads to $T=-I\p{1-\frac{2}{\kappa+1}}$. This is the worst-case example for $\eta\geq1$. Therefore, since the worst case examples in the previous are attained by $\mu$-strongly convex functions with $L$-Lipschitz gradients, we can see that the convergence rate for error $\n{x^{k}-x^{*}}$is sharp.
\end{proof}

\end{document}